\documentclass[
a4paper,
11pt,
twoside,
]{article}
\usepackage[utf8]{inputenc}
\usepackage[T1]{fontenc}
\usepackage[sc]{mathpazo}
\usepackage[english]{babel}
\usepackage{geometry}
\usepackage{amsmath,amsfonts,amssymb,amsthm}
\usepackage[final]{hyperref}
\usepackage{graphicx}
\usepackage{mathtools}
\usepackage[dvipsnames]{xcolor}
\usepackage{pdflscape}
\usepackage{etoolbox}
\usepackage{fancyhdr}
\usepackage{lastpage}
\usepackage{ifdraft}
\usepackage{hyphenat}
\ifdraft{\usepackage{showlabels}}{}
\usepackage{enumitem}
\usepackage{siunitx}
\usepackage{comment}
\usepackage[mathlines,pagewise]{lineno}

\theoremstyle{plain}
\newtheorem{theorem}{Theorem}[section]

\newtheorem{proposition}[theorem]{Proposition}

\newtheorem{remark}[theorem]{Remark}

\numberwithin{equation}{section}
\numberwithin{table}{section}
\numberwithin{figure}{section}

\usepackage{subfigure}

\newcommand{\RR}{\mathbb{R}}

\begin{document}

\setlength{\headheight}{14pt}
\fancypagestyle{mypagestyle}{%
  \renewcommand{\headrulewidth}{0.1pt}%
  \renewcommand{\footrulewidth}{0.1pt}%
  \fancyhf{}%
  \fancyhead[LO,RE]{\scriptsize Nonuniqueness in discontinuous dynamical systems}%
  \fancyhead[RO,LE]{\scriptsize A. And\`o, R. Edwards and N. Guglielmi}%
  \fancyfoot[LE,RO]{\scriptsize \thepage\,/\,\pageref*{LastPage}}%
}
\pagestyle{mypagestyle}

\title{Nonuniqueness phenomena in discontinuous dynamical systems and their regularizations  
}

\author{A. And\`o \thanks{Gran Sasso Science Institute, Viale Francesco Crispi 7, 67100 L'Aquila AQ, Italy (\texttt{alessia.ando@gssi.it}, \texttt{nicola.guglielmi@gssi.it}).}
        \and R. Edwards\thanks{Dept. of Mathematics and Statistics, University of Victoria, P.O. Box 1700 STN CSC, Victoria, B.C., Canada, V8W 2Y2
  (\texttt{edwards@uvic.ca}).}
        \and N. Guglielmi\footnotemark[1]}

\date{\today}

\maketitle

\begin{abstract}
\noindent In a recent paper by Guglielmi and Hairer [\emph{SIAM J. Appl. Dyn. Syst.}, 14 (2015), pp. 1454-1477], an analysis in the $\varepsilon\to 0$ limit was proposed of regularized discontinuous ODEs in codimension-2 switching domains; this was obtained by studying a certain 2-dimensional system describing the so-called hidden dynamics. In particular, the existence of a unique limit solution was not proved in all cases, a few of which were labeled as ambiguous, and it was not clear whether or not the ambiguity could be resolved. 
In this paper, we show that it cannot be resolved in general. A first contribution of this paper is an illustration of the dependence of the limit solution on the form of the switching function. Considering the parameter dependence in the ambiguous class of discontinuous systems, a second contribution is a bifurcation analysis, revealing a range of possible behaviors. Finally, we investigate the sensitivity of solutions in the transition from codimension-2 domains to codimension-3 when there is a limit cycle in the hidden dynamics.
\end{abstract}


\smallskip

\section{Introduction}

   Piecewise smooth ordinary differential equations (ODEs) find applications in fields as diverse as control engineering, social and financial sciences, gene regulatory networks \cite{Acary2008,diBernardo2008,EdGlass2000,MP2011}, and many others. Resorting to them is  not only necessary to, for example, capture adequately, e.g., phenomena which are regulated by switching devices or actions that are triggered as soon as a continuous variable reaches a certain threshold, but also, as is often the case with gene regulatory networks \cite{DeJong2004a, DeJong2004b, Glass2018, gp78a, Grognard2007, Ropers2006}, to replace prohibitively stiff continuous systems whose numerical integration would otherwise be unfeasible, or to capture qualitative dynamics of systems whose functional forms and parameters cannot easily be determined by experiment.
   
   However, integrating discontinuous problems of this kind inevitably entails difficulties due to the potential nonuniqueness of the solution close to the intersection of two or more discontinuity surfaces. Additionally, it can be argued, as is discussed by di Bernardo and coauthors~\cite[p.1]{diBernardo2008}, that physical phenomena are often in fact intrinsically smooth, and using discontinuous models in order to describe them means introducing an artifact anyway. In view of the above, it is sound to approach the numerical integration by first \emph{regularizing} the problems, {\em i.e.}, suitably replacing the jump discontinuities by continuous transitions, and then considering the solution obtained as a certain regularization parameter tends to 0, {\em i.e.}, as the regularized problem approaches the original one. 
   One of the main aims of this work is to emphasize how crucial the choice of a suitable regularization function can be in determining the solution of the switching system. In other words, the problem of nonuniqueness cannot be solved completely.

  The issue has been raised previously, e.g., in \cite{as1998,Dieci2013,gh2015}, where it was observed that the Filippov approach in the most general sense \cite{Filippov} no longer suffices to determine the solution
of the switching system as a limit of a regularization upon entering the intersection of two discontinuity hypersurfaces, since the conditions to stay in such an intersection are two but a convex combination of four vector fields has three degrees of freedom.  Indeed, there are many approaches to characterizing the {\em hidden dynamics}~\cite{Jeffrey2014}, and the most appropriate approach will to some extent depend on the context. See reference~\cite{Jeffrey2022} for a recent summary of these issues. For many problems, considering the discontinuity to be the limit of a smooth monotone regularization is most appropriate.

In order to study the behavior of the solution in codimension-2 discontinuity domains, Guglielmi and Hairer \cite{gh2015} studied, therefore, a bilinear regularization (also called \emph{blending}~\cite{as1998}) and described the hidden dynamics in such domains by applying the technique of asymptotic expansions for singularly perturbed problems to the regularized differential equations, where
the regularization is monotonic. Indeed, in most cases, this approach allows one to predict the behavior of the regularization limit when the Filippov approach gives multiple solutions of the discontinuous problem after entering a codimension-2 domain. However, in a few cases, represented by \cite[(a3) and (a4) in Theorem 6.2]{gh2015}, even the singular perturbation approach is unable to provide a definitive answer, and resorting to the numerical integration inside the switching domain seems a necessary step. The present work aims to clarify the meaning of nonuniqueness in these ambiguous cases, {\em i.e.}, to show that the hidden dynamics depends indeed on the chosen switching function.

   In the particular context of modelling gene-regulatory networks, similar issues have been described. Dynamics at intersections of threshold hyperplanes cannot be avoided, as, for example, when one gene product is confined to its threshold (in a \emph{black wall}) while another slides until it also reaches its threshold, or when two variables experience a damped oscillation around their threshold intersection and converge to it in finite time~\cite{Edwards2000}. Similar finite-time convergence can occur for damped oscillations of three or more gene products, leading to threshold intersections of three or more variables. 

   Several studies in the gene network context have explored the difficulties of the usual Filippov method, using the convex hull of the vector fields adjacent to an intersection of discontinuity surfaces, \emph{versus} alternative approaches with a more restrictive definition of Filippov solution, or a singular perturbation approach, in which the discontinuity is seen as a limit of smooth regulatory functions (usually Hill functions)~\cite{pk2005}. The bilinear, or blended, regularization mentioned above is a special case of a \emph{Filippov solution in the narrow sense}~\cite{MP2011}, which corresponds more closely, but not exactly, to solutions of the singular perturbation limit~\cite{MEV2013a}.

   Furthermore, solutions of systems with nontrivial hidden dynamics (not simply convergence to a stable fixed point) can be handled by an extension of the singular perturbation approach~\cite{Artstein2002}. However, even then it was observed in an example that nonuniqueness can arise again at a quadruple threshold intersection when there is a limit cycle in the hidden dynamics of three variables~\cite{MEV2013b}. This nonuniqueness in the discontinuous limit hides a sensitive dependence that exists in steep Hill-function perturbations from the limit, in which the basins of attraction of different attractors are densely interwoven (infinitely densely in the limit). Another objective of the current paper is to explore this type of behavior in a more general class of switching systems (not necessarily Hill functions) in which there is periodic oscillation in only two switching variables.

   Since we are drawing on ideas from the study of gene networks, where Hill functions are natural, to deal with switching systems of a more general class in the prior work of Guglielmi and Hairer~\cite{gh2015,gh2022,gh2023}, it is helpful to identify and understand the implications of the modeling choices made. Guglielmi and Hairer deal with the class of switching functions whose derivative has compact support, typically with value $-1$ or $1$ outside a small $\varepsilon$-interval around $0$, usually a ramp function, though it might also be a function like $\frac 12 \left[3\left(\frac{x}{\varepsilon}\right)-\left(\frac{x}{\varepsilon}\right)^3\right]$ \cite{gh2015,gh2023}. On the other hand, the gene network literature usually deals with Hill functions, which are strictly increasing on $[0,\infty)$, though a steep Hill function is close to $0$ or $1$ outside a small interval around its threshold. Whether the range of the switching function is $[-1,1]$ or $[0,1]$ is just a matter of scaling, but we will consider a more general class of switching functions, including both ramps and Hill functions (but other sigmoids, like $\tanh \left(\frac{x}{\varepsilon}\right)$ or $\frac{2}{\pi}\tan^{-1} \left(\frac{x}{\varepsilon}\right)$, are also possible), and we will show that the choice can have important consequences. Also, the gene network literature usually describes the fast dynamics in the switching region in terms of the output variable of the switching function, while Guglielmi and Hairer use the input variable (on the $\varepsilon$ scale). This is only a change in description, not an essential change in the model, but in the former case, the variables used to describe the switching dynamics cannot escape from the range of the switching function ($[0,1]$, for example), while in the latter case they can, though they may remain on the $\varepsilon$ scale.
   
The paper is organized as follows. In section \ref{s_framework} we describe our target discontinuous problems and how they can be regularized. In section \ref{s_effects} we direct our attention to the dependence of the solution on the switching function. In section \ref{s_limitcycles} we consider the possibilities that arise from a periodic orbit in the hidden dynamics. In particular, in subsection \ref{ss_steeplimit} we demonstrate the possible sensitivity arising from periodic fast dynamics, which leads to nonuniqueness in the steep-switching limit. In subsection \ref{ss_generalhopf} we give a general condition for the existence of a supercritical Hopf bifurcation in the ambiguous cases in \cite{gh2015}, for which we give two examples in subsection \ref{ss_hopfexamples}, in the first of which the qualitative behavior is the same whether the ramp or the Hill functions are used, while in the second example the nature of the bifurcation depends on the choice of the switching function. Finally, in section \ref{s_conclusions} we give some concluding remarks.
%
%

\section{Framework}
\label{s_framework}
As anticipated in the Introduction, we are interested in the behavior of solutions of discontinuous dynamical systems  close to the intersection of two (or more) discontinuity hypersurfaces. In this section, we briefly describe the approach of asymptotic expansions for singularly perturbed problems and the resulting hidden dynamics, while introducing the relevant notation in the case of two hypersurfaces, which can be straightforwardly extended to the case of three or more.

If $n$ is the dimension of our dynamical systems, we assume that the hypersurfaces 
\begin{equation}\label{hypersurface}
\Sigma_\alpha = \{ y\in \mathbb{R}^n \, ; \, \alpha (y) =0 \}\,, \qquad
\Sigma_\beta = \{ y\in \mathbb{R}^n \, ; \, \beta (y) =0 \} ,
\end{equation}
intersect transversally in $\mathbb{R}^n$. In a neighborhood of
$y\in\Sigma:=\Sigma_{\alpha}\cap\Sigma_{\beta}$, these hypersurfaces divide the phase space into four regions, which we denote by $\mathcal{R}^{++} = \{ y \, ; \, \alpha (y) >0,\, \beta (y) > 0 \}$,
$\mathcal{R}^{+-} = \{ y \, ; \, \alpha (y) >0,\, \beta (y) < 0 \}$, and similarly
$\mathcal{R}^{-+}$ and $\mathcal{R}^{--}$.
The discontinuous ODEs will have the form
\begin{equation}\label{ode}
\dot y = \left\{ \begin{array}{ll} 
f^{++} (y)\,, & y\in \mathcal{R}^{++} \\
f^{+-} (y)\,,~~ & y\in \mathcal{R}^{+-} \\
f^{-+} (y)\,, & y\in \mathcal{R}^{-+} \\
f^{--} (y)\,, & y\in \mathcal{R}^{--}\,, 
\end{array} \right.
\end{equation}
where all four right-hand sides are sufficiently smooth. In the following, a solution that (starts from and) remains in $\mathcal{R}^{++}$ (or $\mathcal{R}^{+-},\mathcal{R}^{-+},\mathcal{R}^{--}$) is called a \emph{classical solution}, while if it reaches either $\Sigma_{\alpha}$ or $\Sigma_{\beta}$ within some finite time, then it turns into a \emph{sliding mode}. Specifically, if only one of the hypersurfaces is reached ($\Sigma_{\alpha}$ with $\beta(y)<0$, without loss of generality), then it is a \emph{codimension-1} sliding mode and, as anticipated, this case can be entirely solved using the Filippov approach, which leads to the differential-algebraic equation
\begin{equation}\label{dae1}
\begin{array}{rcl} \dot y &=& \Bigl(\bigl(1+ \lambda_\alpha \bigr)\, f^{+-}(y) +
 \bigl( 1- \lambda_\alpha \bigr)\, f^{--} (y)\Bigr) \Big/ 2 \\[2mm]
 0 &=& \alpha (y) ,
\end{array}
\end{equation}
where $\lambda_{\alpha}\in[-1,1]$. Such an equation can be solved by first obtaining the expression for $\lambda_{\alpha}=\lambda_{\alpha}(y)$ and then using the classical theory for ODEs after substituting this expression in the first equation.

If the solution then reaches $\Sigma$, it is a codimension-2 sliding mode and the blending approach would give
\begin{equation}\label{dae2}
\begin{array}{rcl} \dot y &=&\Bigl(\bigl(1+ \lambda_\alpha  \bigr)\bigl(1+ \lambda_\beta  \bigr)\, f^{++}(y) +
\bigl(1+ \lambda_\alpha  \bigr)\bigl(1- \lambda_\beta  \bigr)\, f^{+-}(y)\\[1mm]
&&~~+~\bigl(1-\lambda_\alpha  \bigr)\bigl(1+ \lambda_\beta  \bigr)\, f^{-+}(y) +
\bigl(1- \lambda_\alpha  \bigr)\bigl(1- \lambda_\beta  \bigr)\, f^{--}(y)\Bigr) \Big/ 4\\[2mm]
 0 &=& \alpha (y) \\[2mm]
 0 &=& \beta (y) .
\end{array}
\end{equation}
However, in this case, $\lambda_{\alpha}(y)$ and $\lambda_{\beta}(y)$ are quadratic curves that can have 0, 1 or 2 intersections in $[-1,1]\times[-1,1]$, implying that different (sliding or classical) solutions can coexist.

With the aim of solving the resulting apparent ambiguity, we consider a \emph{regularization} of the vector field, defined by a scalar \emph{switching function} $\pi$ as a convex combination of the vector fields:
\begin{equation}\label{regularization}
\begin{array}{rcl} \dot y &=&\Bigl(\bigl(1+ \pi (u)\bigr)\bigl(1+ \pi (v)\bigr)\, f^{++}(y) +
\bigl(1+ \pi (u)\bigr)\bigl(1- \pi (v)\bigr)\, f^{+-}(y)\\[1mm]
&&+\bigl(1-\pi (u)\bigr)\bigl(1+ \pi (v)\bigr)\, f^{-+}(y) +
\bigl(1- \pi (u)\bigr)\bigl(1- \pi (v)\bigr)\, f^{--}(y)\Bigr) \Big/ 4\,.
\end{array}
\end{equation}
where $u=\alpha (y)/\varepsilon$, $v=\beta (y)/\varepsilon$, and $\varepsilon$ is the regularization parameter. The simplest class of switching functions among those that are constant outside a narrow interval is the ramp:
\begin{equation}\label{eq:ramp} r(u)=\left\{\begin{array}{ll} -1\quad &\mbox{if } u\leq-1 \\
u\quad &\mbox{if } -1<u<1 \\
1\quad &\mbox{if } u>1\,.
\end{array}\right. \end{equation}
In order to use \eqref{regularization} to determine the equations describing the hidden dynamics, we resort to the asymptotic expansion of the corresponding solution. The key idea is to consider, in addition to the \emph{slow} time $t$, a \emph{fast} time $\tau:=t/\varepsilon$ to describe the behavior of the solution in the region $\{ y \, ; \, |\alpha (y)|\le \varepsilon ,\, |\beta (y) |\le \varepsilon \}$. If such a region is entered at time $t_0(\varepsilon)$, we can assume without loss of generality that $\alpha(y(t_0(\varepsilon)))=-\varepsilon$, which, in turn, together with the implicit function theorem, allows us to write $y\bigl( t_0(\varepsilon )\bigr) = y_\alpha + \varepsilon a_1 + \mathcal{O} (\varepsilon^2)$ for some constants $y_\alpha,a_1$. We make the assumption
\begin{equation}\label{asymptot}
y\bigl( t_0(\varepsilon ) + t\bigr) = y_0 (t) + \varepsilon \bigl( y_1(t) + \eta_0 (t/\varepsilon)\bigr)
+ \mathcal{O} (\varepsilon^2 ) ,
\end{equation}
which gives $y_0(0)=y_{\alpha}$ and $y_1(0)+\eta_0(0)=a_1$, after substituting $t=0$. Moreover, from \eqref{asymptot} we can write
\begin{equation}\label{alpheps}
\frac 1\varepsilon \, \alpha \bigl( y(t_0(\varepsilon )+t)\bigr) = 
\frac 1\varepsilon \, \alpha \bigl( y_0(t)\bigr) + \alpha ' \bigl( y_0(t)\bigr)
\bigl( y_1(t) + \eta_0 (\tau )\bigr) + \mathcal{O} (\varepsilon ) 
\end{equation}
and observe that, in order to avoid the singularity for $\varepsilon\to 0$, the function $y_0$ must satisfy $\alpha(y_0(t))=0$. Similarly, we also get $\beta(y_0(t))=0$. Differentiating \eqref{asymptot} leads to the equality $\dot y\bigl( t_0(\varepsilon ) + t\bigr) = y_0 (t) + \varepsilon  \dot y_1(t) + \eta'_0 (t/\varepsilon)$. By inserting it into \eqref{regularization}, replacing $t\to\varepsilon \tau$, and setting $\varepsilon =0$, we get
\begin{equation}\label{dgleta0}
\begin{array}{rcl} \eta_0 '(\tau ) &\!\!=\!\!&  \dot y(t_0(0)) - \dot y_0 (0)\\[1mm]
&\!\!=\!\!&
\Bigl(\bigl(1+ \pi \bigl(u (\tau )\bigr)\bigr)\bigl(1+ \pi \bigl(v (\tau )\bigr)\bigr)\, f^{++}(y_0) +
\bigl(1+ \pi \bigl(u (\tau )\bigr)\bigr)\bigl(1- \pi \bigl(v (\tau )\bigr)\bigr)\, f^{+-}(y_0)\\[1mm]
&\!\!+\!\!&\bigl(1-\pi \bigl(u (\tau )\bigr)\bigr)\bigl(1+ \pi \bigl(v (\tau )\bigr)\bigr)\, f^{-+}(y_0) +
\bigl(1- \pi \bigl(u (\tau )\bigr)\bigr)\bigl(1- \pi \bigl(v (\tau )\bigr)\bigr)\, f^{--}(y_0)
\Bigr) \Big/ 4\\[2mm]
 &\!\!-\!\!& \dot y_0 (0).
\end{array}
\end{equation}
Replacing $t\to\varepsilon \tau$ and setting $\varepsilon =0$ in \eqref{alpheps}, we can write $u (\tau )= \alpha ' (y_0 ) \bigl( y_1(0) + \eta_0 (\tau )\bigr)$
and $v (\tau )= \beta ' (y_0 ) \bigl( y_1(0) + \eta_0 (\tau )\bigr)$, which are indeed the functions we use to describe the hidden dynamics. The relevant differential equations are obtained by multiplying \eqref{dgleta0} first by $\alpha'(y_0)$ and then by $\beta'(y_0)$, while recalling that $\alpha '(y_0 ) \dot y_0 (0) =\beta '(y_0 ) \dot y_0 (0) =0$:
\begin{equation}\label{dynsyst}
\begin{array}{rcl}
{ u'} &\!\!=\!\!& \Bigl(\bigl(1+ \pi (u)\bigr)\bigl(1+ \pi (v)\bigr)\, f_\alpha^{++} +
\bigl(1+ \pi (u)\bigr)\bigl(1- \pi (v)\bigr)\, f_\alpha^{+-}\\[1mm]
&&~+~\bigl(1-\pi (u)\bigr)\bigl(1+ \pi (v)\bigr)\, f_\alpha^{-+} +
\bigl(1- \pi (u)\bigr)\bigl(1- \pi (v)\bigr)\, f_\alpha^{--}\Bigr) \Big/ 4\\[2mm]
{ v'} &\!\!=\!\!& \Bigl(\bigl(1+ \pi (u)\bigr)\bigl(1+ \pi (v)\bigr)\, f_\beta^{++} +
\bigl(1+ \pi (u)\bigr)\bigl(1- \pi (v)\bigr)\, f_\beta^{+-}\\[1mm]
&&~+~\bigl(1-\pi (u)\bigr)\bigl(1+ \pi (v)\bigr)\, f_\beta^{-+} +
\bigl(1- \pi (u)\bigr)\bigl(1- \pi (v)\bigr)\, f_\beta^{--}\Bigr) \Big/ 4 
\end{array}
\end{equation}
for $f_\alpha^{++} = \alpha ' (y_0 ) f^{++}(y_0)$,$\,\ldots$, $f_\beta^{--} = \beta ' (y_0 ) f^{--}(y_0)$.

In the following we assume that $\alpha(y)=y^{(1)}-\theta_1$ and $\beta(y)=y^{(2)}-\theta_2$ for some constant $\theta_1,\theta_2$, where $y^{(1)}$ and $y^{(2)}$ are the components of the vector $y$. In this case, we have $y_0=(\theta_1,\theta_2)^\intercal$ and $\alpha'(y_0)=\beta'(y_0)=1$, so that $f_\alpha^{++} = f^{++}(y_0)$,$\,\ldots$, $f_\beta^{--} = f^{--}(y_0)$.

Such an assumption is anyway typical in applications, such as in the context of gene network modeling. There, the switching function is often taken to be a Hill function~\cite{pk2005}
\begin{equation}\label{eq:hill}
Z(x,\theta,\varepsilon)=\frac{x^{1/\varepsilon}}{\theta^{1/\varepsilon}+x^{1/\varepsilon}},
\end{equation} with $\varepsilon\in (0,1]$ as the perturbation parameter\footnote{The notation originally introduced in \cite{pk2005} used the letter $q$ to indicate the regularization parameter, and so did subsequent work on the subject of gene regulatory networks. However, we decide to use $\varepsilon$ for consistency throughout the paper, as that is the letter normally used in the context of singular perturbation theory.}. This differs from the class of switching functions represented by the ramp, in that it is not constant outside an interval. Differentiating the Hill function gives 
$\frac{dZ}{dx}=\frac{Z(1-Z)}{\varepsilon x}$
so that $\frac{dZ}{dt}=\frac{Z(1-Z)}{\varepsilon x}\dot{x}$, or in the fast time scale, with $t=\tau \varepsilon$, this becomes 
\[\frac{dZ}{d\tau}=\frac{Z(1-Z)}{x}\dot{x}\,.\]
If we were to define $Z=r(u(x))$ as in \eqref{eq:ramp}, we would have $\frac{dZ}{dx}=\frac{1}{2\varepsilon}$ on $x\in (\theta-\varepsilon, \theta+\varepsilon)$, and $0$ elsewhere, so that $\frac{dZ}{dt}=\frac{1}{2\varepsilon}\dot{x}$ and changing to fast time via $t=2\varepsilon \tau$, $\frac{dZ}{d\tau}=\dot{x}$. Thus, the extra factor $\frac{Z(1-Z)}{x}$ that appears in the $Z$ equations when we use Hill functions comes from the nonlinearity of the Hill function.

\section{Effects of the choice of switching function}
\label{s_effects}
In this section, we show how crucial the choice of a switching function can be to the stability properties of the  equilibria in a codimension-2 region.

Given a system in $\RR^2$, for example, 
\begin{eqnarray*}
 \dot{x}_1 &=& f(r(u),r(v)) = f(\rho,\sigma) \\
 \dot{x}_2 &=& g(r(u),r(v)) = g(\rho,\sigma)
\end{eqnarray*}
and a fixed point at $(\rho^*,\sigma^*)$,
the Jacobian matrix (taking derivatives with respect to $\rho$ and $\sigma$) is 
\[ J=\left[ \begin{array}{cc} \partial_{\rho} f & \partial_{\sigma} f \\
\partial_{\rho} g & \partial_{\sigma} g \end{array}\right],\] 
and this is also the Jacobian with respect to $u$ and $v$ when $\rho=r(u)=u$ and $\sigma=r(v)=v$. 

Now let $\pi$ be some function other than $r$ that is monotonic increasing and smooth ($C^1$) on $[-1,1]$. The fixed points are at $u^*=\pi^{-1}(\rho^*)$, $v^*=\pi^{-1}(\sigma^*)$, which are well-defined because of the monotonicity and continuity of $\pi$. The new Jacobian with respect to $u$ and $v$ is 
\[ J_{\pi}=\left[ \begin{array}{cc} \pi'(u) \partial_{\rho} f & \pi'(v) \partial_{\sigma} f \\
\pi'(u) \partial_{\rho} g & \pi'(v) \partial_{\sigma} g \end{array}\right]
=\left[ \begin{array}{cc} \partial_{\rho} f & \partial_{\sigma} f \\
\partial_{\rho} g & \partial_{\sigma} g \end{array}\right]\left[ \begin{array}{cc} \pi'(u) & 0 \\
0 & \pi'(v) \end{array}\right]=JD.  \] 
Then $\text{det}(J_{\pi})=(\text{det}(J))(\text{det}(D))$ and of course $\text{det}(D)>0$, which gives $\text{sign}(\text{det}(J))=\text{sign}(\text{det}(J_{\pi}))$, while $\text{Tr}(J_{\pi})=\pi'(u)\partial_{\rho} f+\pi'(v)\partial_{\sigma} g$. Thus, properties of the dynamics, for example, the stability of a fixed point, can depend on the choice of switching function through the trace of the Jacobian. In particular, $\text{sign}(\text{Tr}(J_{\pi}))\ne\text{sign}(\text{Tr}(J))$ if $\text{Tr}(J_{\pi})\text{Tr}(J)<0$, {\em i.e.}, when 

\begin{equation} [\pi'(u) \partial_{\rho} f +\pi'(v) \partial_{\sigma} g][\partial_{\rho} f + \partial_{\sigma} g]<0\,.\label{eq:tracechange}
\end{equation}

Note that the $Z_i$ variables in the Hill function formulation could also be re-scaled to the interval $[-1,1]$ instead of $[0,1]$, letting 
\begin{equation}
Z_1=\frac{u+1}{2}\quad\text{and}\quad Z_2=\frac{v+1}{2}\,,
\label{eq:substitutions}
\end{equation}
to get
\begin{eqnarray*} u' &=& (1-u^2)f(u,v) \\
v' &=& (1-v^2)g(u,v)\,,
\end{eqnarray*}
where $f(u,v)$ is $\dot{x}_1$ and $g(u,v)$ is $\dot{x}_2$ with the substitutions~\eqref{eq:substitutions}.

\subsection{Dependence of behavior on the switching function: an illustrative example}

In this section, we consider an example where a fixed point in the codimension-2 region is stable with the Hill function, but unstable with the ramp function.

Let
\begin{eqnarray*}
\dot{x}_1&=&(1-s^+(x_1,\theta_1)+s^+(x_1,\theta_1)s^+(x_2,\theta_2))k_1-\gamma_1 x_1 \\
\dot{x}_2&=&(1-s^+(x_1,\theta_1)-s^+(x_2,\theta_2)+2s^+(x_1,\theta_1)s^+(x_2,\theta_2))k_2 - \gamma_2 x_2\,,
\end{eqnarray*}
where $s^+(x,\theta)=\left\{ \begin{array}{ll} 0 & \text{ on }x<\theta \\ 1 & \text{ on } x>\theta\end{array}\right. \,.$
The dynamics in the codimension-2 region using Hill functions defined as \eqref{eq:hill} where $\lim_{\varepsilon\to 0}Z(x,\theta,\varepsilon)= s^+(x,\theta)$, after changing variables and time scales, are
\begin{eqnarray*}
Z'_1&=&\frac{Z_1(1-Z_1)}{\theta_1}[(1-Z_1+Z_1Z_2)k_1-\gamma_1\theta_1] \\
Z'_2&=&\frac{Z_2(1-Z_2)}{\theta_2}[(1-Z_1-Z_2+2Z_1Z_2)k_2 - \gamma_2\theta_2]\,,
\end{eqnarray*}
but with the ramp functions, using the variables $u=\frac{x_1-\theta_1}{\varepsilon}, v=\frac{x_2-\theta_2}{\varepsilon}$, and $\pi=r$ as in \eqref{eq:ramp} so that $\lim_{\varepsilon\to 0} \frac 12 (1+\pi(u(x_1)))=s^+(x_1,\theta_1)$, and similarly for $\pi(v)$, we get
\begin{eqnarray*}
u'&=&\frac{1}{4}(uv-u+v+3)k_1-\gamma_1\theta_1 \\
v'&=&\frac{1}{2}(uv+1)k_2 - \gamma_2\theta_2
\end{eqnarray*}
in $[-1,1]^2$.

With $\phi_1=\frac{\gamma_1\theta_1}{k_1}=\frac 12$ and $\phi_2=\frac{\gamma_2\theta_2}{k_2}=\frac{7}{16}$, the fixed point in the codimension-2 region is $(Z_1^*,Z_2^*)$, or rescaled for $u,v$ variables, $(u^*,v^*)=(2Z_1^*-1,2Z_2^*-1)$ independent of the regularization, apart from the scaling, with $Z_1^*=\frac{23-\sqrt{17}}{32}$ and $Z_2^*=\frac{9-\sqrt{17}}{32}$, or approximately $(Z_1^*,Z_2^*)=(0.5899,0.4100)$. With $k_1=1, k_2=5$, the fixed point is unstable for the ramp regularization, since the Jacobian matrix for the ramp regularization reads
\[
J_r(u,v)=\frac{1}{4}\cdot\begin{pmatrix}
v-1 & u+1 \\
10v & 10u 
\end{pmatrix},
\]
so that $ \text{det}(J_r(u^*,v^*)) = -\frac{5}{2}(u^*+v^*)= \frac{5\sqrt{17}}{16}>0$ and $ \text{tr}(J_r(u^*,v^*)) = v^*-1+10u^*= \frac{47-11\sqrt{17}}{16}>0$. Indeed, as shown by numerical simulations, solutions entering the codimension-2 region on the black wall at $(u,v)=(1-2\phi_1,-1)$ exit into the classical solution in $(u,v)=(1,1)$, the positive quadrant. 
However, the fixed point is stable for the Hill function regularization, since the  Jacobian matrix at the fixed point in this case when, for instance, $\theta_1=\theta_2=1$, is given by
\[
J_H(Z_1^*,Z_2^*)=\begin{pmatrix}
Z_1^*(1-Z_1^*)(Z_2^*-1)& Z_1^*(1-Z_1^*)\cdot Z_1^* \\
5Z_2^*(1-Z_2^*)(2Z_2^*-1) & 5Z_2^*(1-Z_2^*)(2Z_1^*-1)
\end{pmatrix},
\]
so that $ \text{det}(J_H(Z_1^*,Z_2^*))=\frac{Z_1^*Z_2^*(1-Z_1^*)(1-Z_2^*)}{2}\text{det}(J_r(u^*,v^*))>0$ and
\begin{eqnarray*}
\text{tr}(J_H(Z_1^*,Z_2^*)) &=& (1-Z_2^*)(-Z_1^*(1-Z_1^*)+5Z_2^*(2Z_1^*-1))\\
&=& \frac{173-53\sqrt{17}}{512}<0\,.
\end{eqnarray*}
Indeed, numerical simulations show that codimension-2 solutions entering the co-dimension 2 region on the black wall at $(Z_1,Z_2)=(1-\phi_1,0)$ approach $(Z_1^*,Z_2^*)$.
\subsection{Singular perturbation using an arbitrary monotonic switching function}
Recall the condition \eqref{eq:tracechange}, which is necessary for stability change. Consider now a class of monotonic functions, ${\cal S}$, such that $\pi\in{\cal S}$ satisfies these four properties:
\begin{equation}\label{eq:S}\left\{\begin{array}{lr}
\pi(u) &\text{is odd}\\
\pi'(u)>0& \text{on } (-1,1)\\
\pi''(u)\le 0 &\text{on } (0,1)\\
\pi(u)=1 &\text{on } [1,\infty)\,.
\end{array}
\right.
\end{equation}
\begin{proposition}
The stability of a fixed point $(\rho^*,\sigma^*)\in(-1,1)^2$ in codimension-2 when using the ramp function \eqref{eq:ramp} for both variables is different from the one obtained when using $\pi$ if

\begin{equation}
\begin{array}{rcl}
\displaystyle
\left. \frac{\partial_\rho f}{\partial_\sigma g}\right|_{(\rho^*,\sigma^*)} &\in& 
\displaystyle
\left( -1,-\frac{\pi'(v^*)}{\pi'(u^*)} \right)\text{ if } \pi'(v^*)<\pi'(u^*)\,,\\[5mm]
\displaystyle 
\left. \frac{\partial_\rho f}{\partial_\sigma g}\right|_{(\rho^*,\sigma^*)} &\in& 
\displaystyle
\left( -\frac{\pi'(v^*)}{\pi'(u^*)},-1 \right)\text{ if } \pi'(v^*)>\pi'(u^*)\,, 
\end{array}
\label{eq:stabilitychange}
\end{equation}
but not outside $[-1,1]^2$ (and we say nothing about the endpoints).
\end{proposition}
\begin{proof}
This comes from ~\eqref{eq:tracechange} above, as follows. Multiply the two terms in ~\eqref{eq:tracechange} and divide by $(\partial_\sigma g)^2$ to get 
\[\pi'(u)\left( \frac{\partial_\rho f}{\partial_\sigma g}\right)^2+(\pi'(u)+\pi'(v))\left( \frac{\partial_\rho f}{\partial_\sigma g}\right)+\pi'(v)<0\,,\]
which is of the form $ax^2+(a+c)x+c<0$, with $a>0, c>0$. The two roots of this quadratic are $x=-1$ and $-\frac ca$. Thus, $x\in \text{range}\left\{-1,-\frac ca\right\}$, the order of these two terms depending on the magnitude of $\frac ca$. 
\end{proof}

\begin{proposition}
Given $f$ and $g$, and fixed point $(\rho^*,\sigma^*)$, the condition in~\eqref{eq:stabilitychange} can be achieved for some function $\pi\in{\cal S}$ if and only if the following conditions are satisfied:
\begin{enumerate}
\item $\text{det}(J(\rho^*,\sigma^*))>0$, and
\item $\left. \frac{\partial_\rho f}{\partial_\sigma g}\right|_{(\rho^*,\sigma^*)}  <0$, and either 
\item\begin{enumerate}\item
$  \left. \frac{\partial_\rho f}{\partial_\sigma g}\right|_{(\rho^*,\sigma^*)} < -1 \text{ and } |\rho^*|>|\sigma^*| $\,, or
\item  $\left. \frac{\partial_\rho f}{\partial_\sigma g}\right|_{(\rho^*,\sigma^*)} > -1 \text{ and }  |\rho^*|<|\sigma^*|\,. $
\end{enumerate}
\end{enumerate}
\end{proposition}

\begin{proof}
The first condition must be satisfied, since otherwise the fixed point is always a saddle point, and its stability cannot be changed by modifying the trace. The second condition must be satisfied, because otherwise the sign of $\text{Tr}(J)$ cannot be changed when $J$ is left multiplied by $D$.

Given the first two conditions, we now show that the third is necessary and sufficient.
By \eqref{eq:S}, $\pi'(-x)=\pi'(x)$ for all $x$, and recall that $\rho^*=\pi(u^*)$ and $\sigma^*=\pi(v^*)$.

Suppose first that $|\rho^*|>|\sigma^*|$. Then $\pi'(u^*)\le\pi'(v^*)$, or, letting $\eta=\pi^{-1}$, $\eta'(\rho^*)\ge\eta'(\sigma^*)$. The fact that $\pi''\le 0$ on $(0,1)$ ($\pi$ is concave down or at least flat), and therefore $\rho''\ge 0$ on $(0,1)$, leads to the following inequalities:  
\[ \frac{\eta(|\sigma^*|)}{|\sigma^*|} \le \eta'(\sigma^*) \le \frac{\eta(|\rho^*|)-\eta(|\sigma^*|)}{|\rho^*|-|\sigma^*|}\le \eta'(\rho^*)\le\frac{1-\eta(|\rho^*|)}{1-|\rho^*|}\,.\]
Thus, 
\[ 1\ge \frac{\eta'(\sigma^*)}{\eta'(\rho^*)}\ge \frac{\eta(|\sigma^*|)}{|\sigma^*|}\frac{(1-|\rho^*|)}{(1-\eta(|\rho^*|)}\,. \]
The last term on the right $\to 0$ as $\rho(|\sigma^*|)\to 0$ (for fixed $\sigma^*$, of course). It is always possible to choose a $\pi$ (and therefore $\eta$) to make $\eta(|\sigma^*|)$ as close to $0$ as we like. 

Then, since $\pi'(v^*)=\frac{1}{\eta'(\sigma^*)}$ and $\pi'(u^*)=\frac{1}{\eta'(\rho^*)}$, we have 
\begin{equation} 
\frac{\pi'(v^*)}{\pi'(u^*)}=\frac{\eta'(\rho^*)}{\eta'(\sigma^*)}\,,
\label{eq:rhoratio}
\end{equation}
which can then take any values in $[1,\infty)$ by appropriate choice of $\pi\in{\cal S}$.

Then if (and only if) $\left. \frac{\partial_\rho f}{\partial_\sigma g}\right|_{(\rho^*,\sigma^*)}<-1$, we can choose $\pi$ so that it falls in the range in~\eqref{eq:stabilitychange}.

Now suppose that $|\rho^*|<|\sigma^*|$. The argument above applies with the roles of $\rho$ and $\sigma$ reversed, so
\[ 1\ge \frac{\eta'(\rho^*)}{\eta'(\sigma^*)}\ge \frac{\eta(|\rho^*|)}{|\rho^*|}\frac{(1-|\sigma^*|)}{(1-\eta(|\sigma^*|)} \]
and again, the last term on the right $\to 0$ as $\eta(|\rho^*|)\to 0$ (for fixed $\rho^*$), and any value above $0$ is achievable by appropriate choice of $\pi\in{\cal S}$.
Thus, the ratio in Equation~\eqref{eq:rhoratio} can take any values in $(0,1]$.

Then if (and only if) $\left. \frac{\partial_\rho f}{\partial_\sigma g}\right|_{(\rho^*,\sigma^*)}>-1$, we can choose $\pi$ so that it falls in the range in~\eqref{eq:stabilitychange}.
\end{proof}

\begin{remark}\label{kepsilon}
When the fast equations define variables which in fact evolve at (slightly) different rates, it is possible to pick correspondingly two switching functions which share the same shape but differ in steepness. For instance, the switching $u$ variable can be a ramp function with slope $\frac{1}{\varepsilon}$ while that for $v$ has slope $\frac{\kappa}{\varepsilon}$. In the hidden dynamics, this would correspond to multiplying the right-hand side. of $v'$ by $\kappa$, which, as it will be made clearer by the examples shown in the next sections, plays an important role in the resulting dynamical analysis. Similarly, in the case of Hill functions, it is possible to use a different time scale for either of the variables by defining $Z(x_i,\theta_i,\varepsilon)=\frac{x_i^{\kappa/\varepsilon}}{\theta_i^{\kappa/\varepsilon}+x_i^{\kappa/\varepsilon}}$, obtaining thus $\frac{dZ_i}{d\tau}=\kappa\frac{Z_i(1-Z_i)}{x_i}\dot x_i$, which again corresponds to multiplying the right-hand side by $\kappa$. We note that in \cite{pk2005} it was already pointed out that when different variables $x_i$ have different values $\varepsilon_i$ (e.g., one has $\varepsilon_i$, another has $\kappa \varepsilon_i$), the dynamics may be different in comparison to uniform $\varepsilon_i$'s even in the limit as all $\varepsilon_i\to \infty$.
\end{remark}
\section{Limit cycles in a codimension-2 region} 
\label{s_limitcycles}
In this section, we consider existence of Hopf bifurcations and stable periodic orbits in co-dimension 2 domains. 

For example, we consider $\pi(u)=u$ and $\pi(v)=\kappa v$ as switching functions. To find the first Lyapunov coefficient $a$ in the Hopf theorem \cite[pp.151--152]{gh1983}, one needs to make a change of variables to put the system into canonical form \cite[equation (3.4.10)]{gh1983}.

In a modified version of a system studied by Del Buono, Elia, and Lopez \cite[equation (3.5)]{DelBuono2014}, 
\begin{eqnarray*}
Z'_1&=&-2Z_1-Z_2+\frac{23}{10}+\delta\left(Z_1-\frac{9}{10}\right)\left(Z_2-\frac 12\right) \\
Z'_2&=&8Z_1+\mu\left(Z_2-\frac 12\right)-\frac{36}{5}\,,
\end{eqnarray*}
with fixed point at $(\frac{9}{10},\frac 12)$, as in the original Del Buono system, the Hopf parameter $a=-\frac{\delta^2}{2}$, which means that the inclusion of the term with parameter $\delta\ne 0$ guarantees a supercritical Hopf bifurcation at $\mu=2$, where eigenvalues of the Jacobian pass through the imaginary axis transversally with positive velocity $d=\frac{d}{d\mu}\text{Re}(\lambda)=\frac 12$.

With the ramp regularization, the same bifurcation occurs at $\mu=2$ in the differently modified Del Buono system (also rescaled to the codimension-2 region, $[-1,1]^2$):
\begin{eqnarray*}
u' &=& -2u-v+\frac 85 +\frac v2\left(u-\frac 45\right) \\
v' &=& \left[8u+\frac{\mu}{\kappa}v-\frac{32}{5}\right]\kappa\,.
\end{eqnarray*}
With $\kappa=3.3$, $\mu=2.005$, there is a periodic orbit inside the co-dimension 2 region $[-1,1]^2$.

\subsection{Nonuniqueness arising from periodic fast dynamics}
\label{ss_steeplimit}
We embed this in a system in dimension 3, similar but not identical to the original system of Del Buono (rescaled), as follows:
\begin{eqnarray}
u' &=& -2u-v+\frac 85 +\frac v2\left(u-\frac 45\right) +\frac 12 \left(\frac{1+w}{2}\right) \nonumber \\
v' &=& \left[8u+\frac{\mu}{\kappa}v-\frac{32}{5}+ pv\left(\frac{1+w}{2}\right) \right]\kappa \label{eq:codim3}\\
w' &=& \frac 23 +\frac{uv}{2}\,, \nonumber
\end{eqnarray}
for $u,v,w\in [-1,1]^3$ (and of course $u,v,w$ are each replaced by $-1$ or $1$ outside this region in the equations above), and using $\mu=2.005$, $\kappa=3.3$ as before, and also $p=2.1$.
Then starting from $w\ll-1$ ({\em i.e.}, the large scale variable $x_3<0$), the $u,v$-system is as before and converges to the limit cycle. But $w$ increases until the solution enters the codimension-3 region $[-1,1]^3$. In the limit $\varepsilon\to 0$, this happens everywhere on the limit cycle simultaneously. For small nonzero $\varepsilon$, the oscillations in $(u,v)$ occur on the fast time scale and so are very fast in relation to the (slow) motion of $x_3$. So the entry point on the $w=-1$ face of the codimension-3 region, while determined, is very sensitive to parameters, and to initial conditions. From some points on the limit cycle at $w=-1$, we go to a classical solution with $u>1, v<-1, w>1$ ({\em i.e.}, $x_1>\theta_1,\, x_2<\theta_2,\, x_3>\theta_3$), and from other points on the limit cycle we go to a co-dimension 1 solution with $v>1, w>1$, and $u\in(-1,1)$ (so $x_1=\theta_1+{\cal O}(\varepsilon),\, x_2>\theta_2,\, x_3>\theta_3$). See Figure~\ref{fig:codim3}.

\begin{figure}
    \begin{center}
    \subfigure[]
 {    \includegraphics[width=6cm]{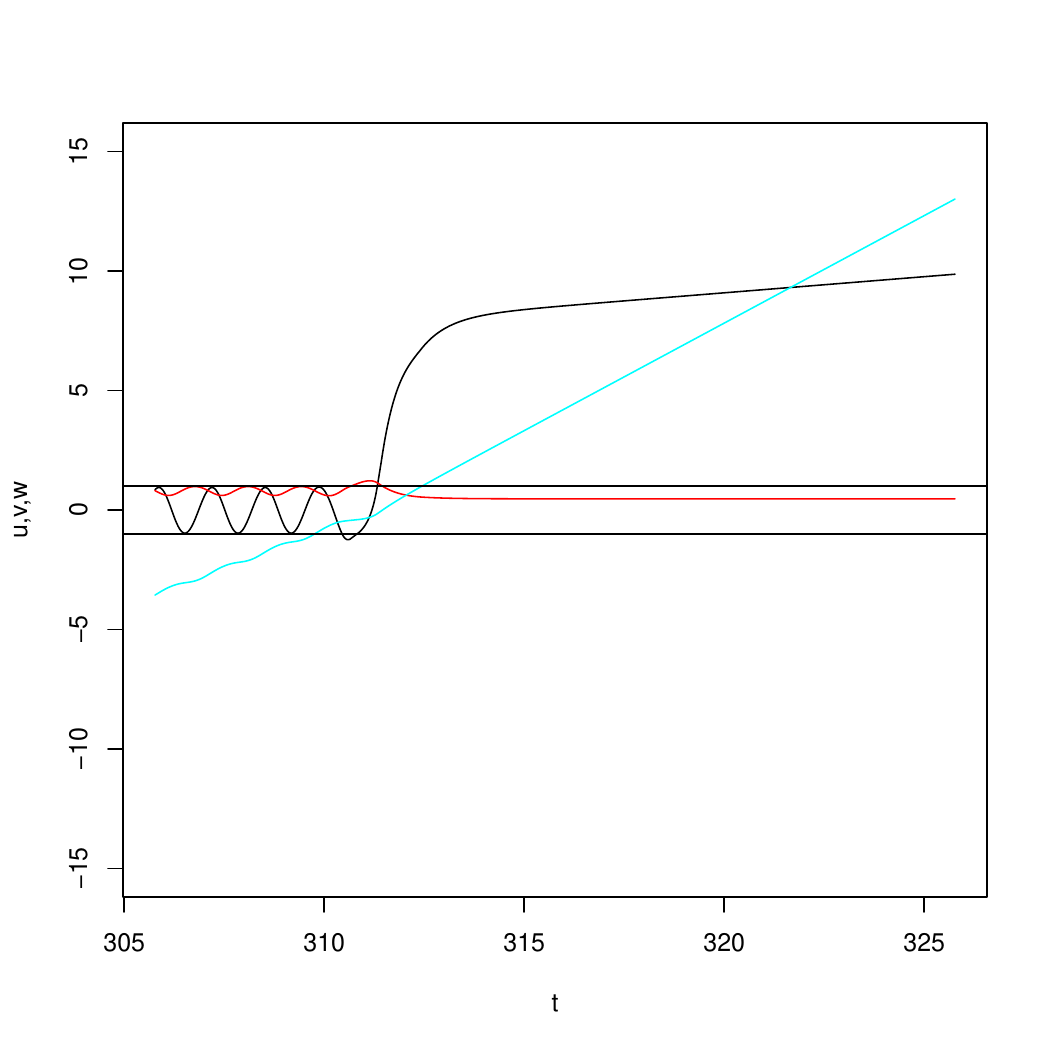} }   \hspace{.5cm} \subfigure[]
 {\includegraphics[width=6cm]{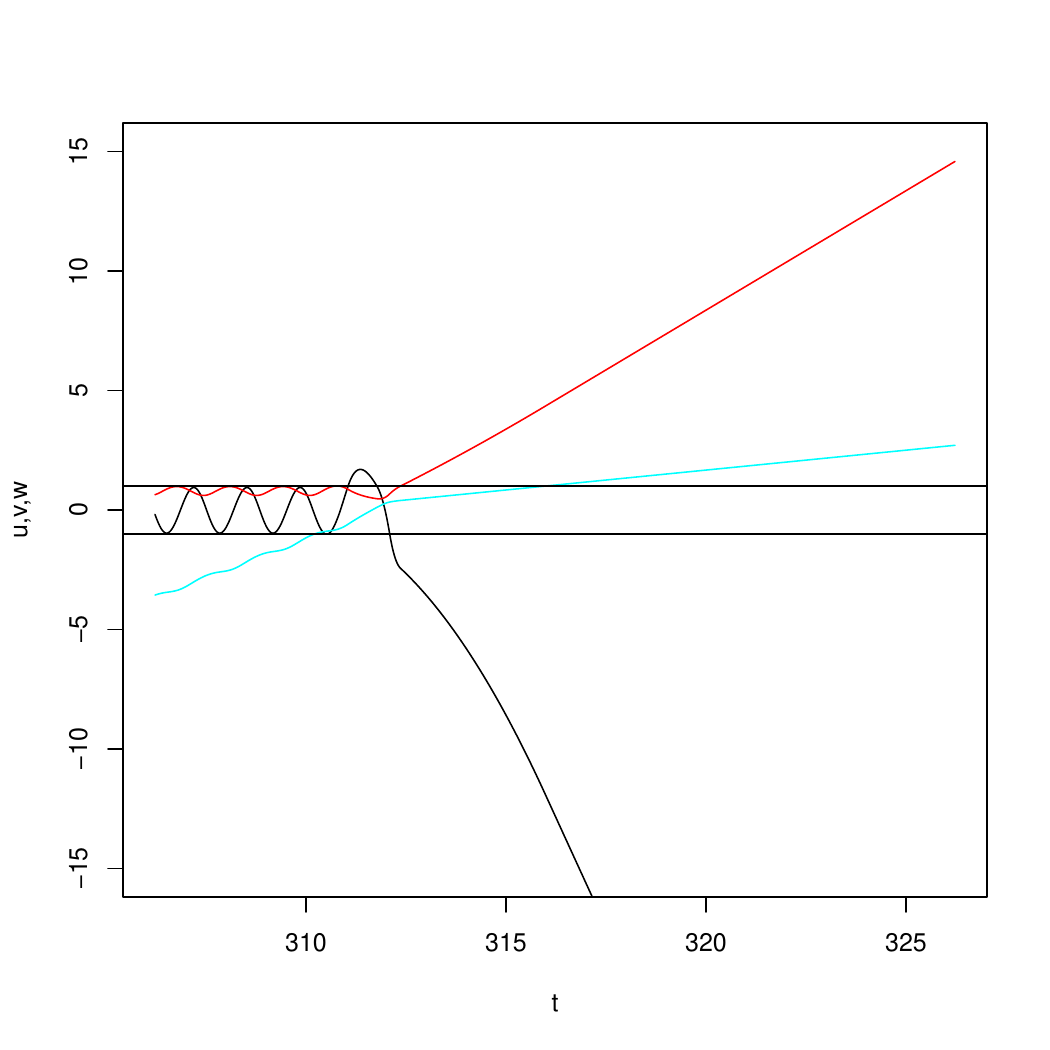} }
  \subfigure[]
 {  \includegraphics[width=6cm]{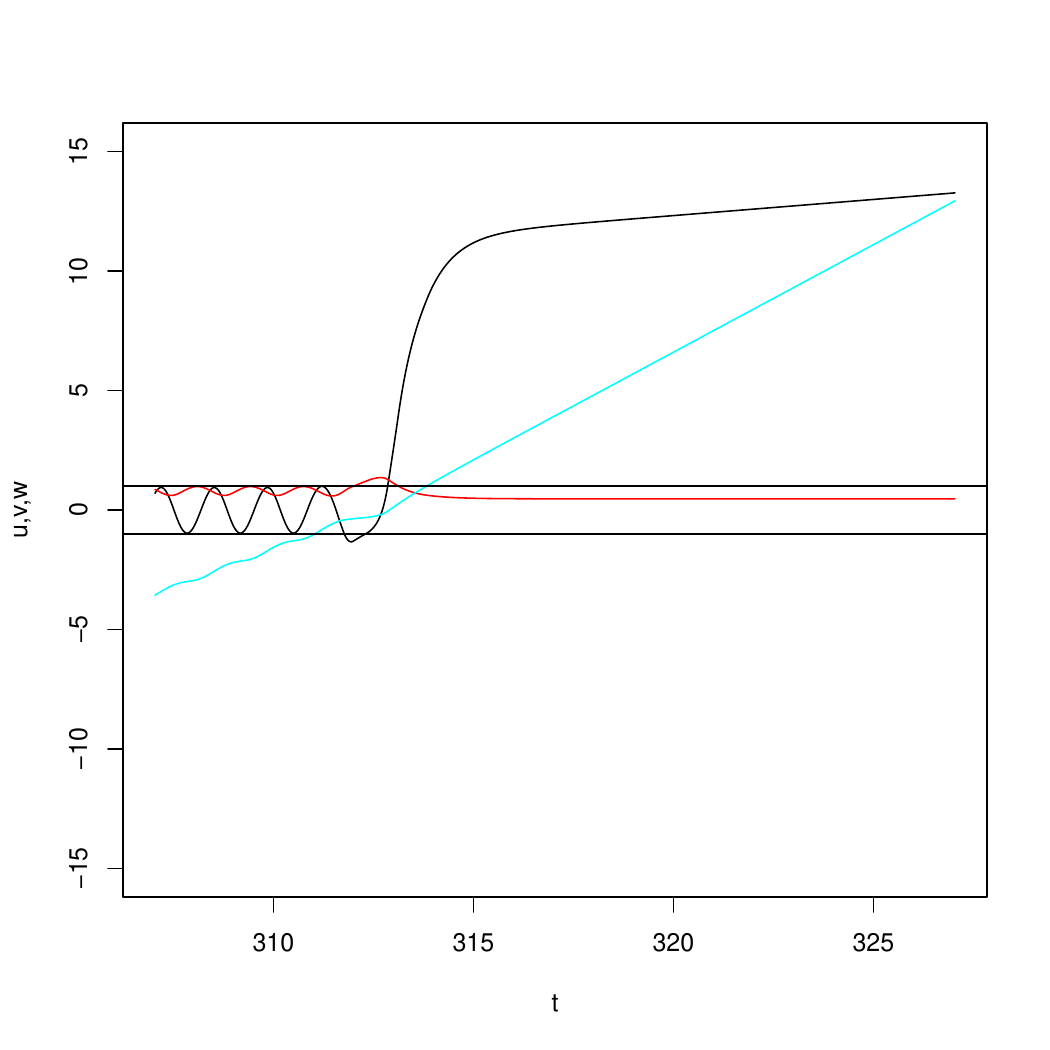}}\hspace{.5cm} \subfigure[]
 {\includegraphics[width=6cm]{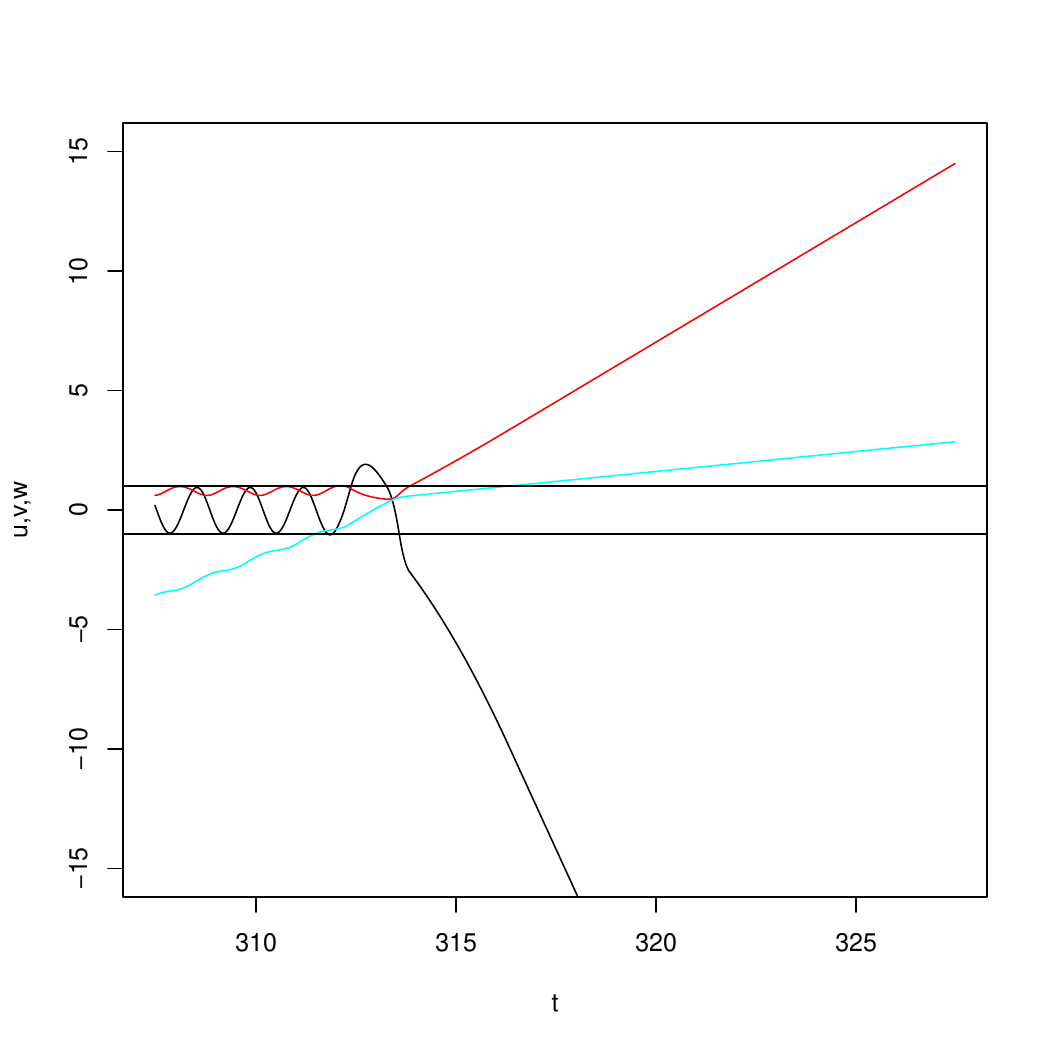}} 
    \caption{Simulations of equations~\eqref{eq:codim3}, with $u(0)=0.8, v(0)=0.9$ and (a) $w(0)=-199.4$, (b) $w(0)=-199.8$, (c) $w(0)=-200.2$, (d) $w(0)=-200.6$. This shows the sensitivity of the final trajectory with respect to the initial point. In (a) and (c) the trajectory exits the threshold intersection with $u$ going up and $v$ remaining at the threshold, while in (b) and (d), $u$ goes down and $v$ goes up.
    \label{fig:codim3}}
    \end{center}
\end{figure}

This exemplifies the sensitivity (nonuniqueness in the limit) of future solutions that can arise when the fast variables do not converge to a fixed point.
It should be noted that for the singularly perturbed problem the same sensitivity occurs with small changes in $\varepsilon$. The type of sensitive dependence found here was identified in examples in prior work~\cite{Edwards2018,MEV2013b}, but there the transition was from a co-dimension 3 region to a co-dimension 4 region.

\subsection{General condition for supercritical Hopf bifurcations}
\label{ss_generalhopf}
Stable limit cycles can appear in the hidden dynamics when a supercritical Hopf bifurcation occurs as a parameter is changed. For a general 2-dimensional multilinear system, we seek conditions on parameters of the equations for a Hopf bifurcation to occur and for it to be supercritical (so we get a stable periodic orbit when we perturb parameters in one direction or the other).

Consider the general 2-dimensional system
\begin{eqnarray*}
    u' &=& a_1 uv + b_1 u + c_1 v +d_1 \\
    v' &=& a_2 uv + b_2 u + c_2 v +d_2\,.
\end{eqnarray*}
Then the nullclines intersect where
\begin{equation} -v=\frac{b_1u+d_1}{a_1u+c_1}=\frac{b_2u+d_2}{a_2u+c_2}\,,\label{eq:fixedpoint}\end{equation}
so a fixed point occurs when
\[ -1<\frac{b_1u^*+d_2}{a_1u^*+c_1}=\frac{b_2u^*+d}{a_2u^*+c}<1 \quad\text{for}\quad -1<u^*<1\,. \]
The Jacobian of the system at $(u^*,v^*)$ is
\[ J^*=J(u^*,v^*)=\left[\begin{array}{cc} a_1v^*+b_1 & a_1u^*+c_1\\
a_2v^*+b_2 & a_2u^*+c_2 \end{array} \right]=\left[\begin{array}{cc} p & q\\
r & s \end{array} \right]\,, \]
where this expression defines $p,q,r,s$.

A Hopf bifurcation occurs when $\text{Tr}(J^*)=0$ and $\text{Det}(J^*)>0$, {\em i.e.}, where $p=-s$ and $ps-qr>0$.

\begin{proposition} Such a Hopf bifurcation is supercritical when
\[ \frac{(b_1c_1-a_1d_1)}{a_1v^*+b_1}[a_1^2(a_2d_2-b_2c_2)+a_2^2(b_1c_1-a_1d_1)]>0\,. \]
\end{proposition}
\begin{proof}
At a Hopf bifurcation, eigenvalues are $\lambda=\pm i\omega$, and $\omega^2=\text{det}(J^*)$, since $\text{Tr}(J^*)=0$. Eigenvalues are given by 
\[ \left[\begin{array}{cc} p+i\omega & q\\
r & s+i\omega \end{array} \right] \xi ={\bf 0}\,, \]
so, for example, we can take $\xi=\left[\begin{array}{c} q \\ -p\end{array} \right]+i\left[\begin{array}{c} 0 \\ -\omega\end{array} \right]$. In order to translate coordinates to get the system into the canonical form for the Hopf theorem, we use the change-of-coordinates matrix
\[ T=\left[\begin{array}{cc} q & 0 \\ -p & -\omega\end{array} \right],\quad\text{with}\quad T^{-1}=-\frac{1}{\omega q}\left[\begin{array}{cc} -\omega & 0 \\ p & q\end{array} \right]\quad\text{and}\quad D=\left[\begin{array}{cc} 0 & -\omega \\ \omega & 0\end{array} \right] \]
so that $JT=TD$. Now our new coordinates $(x,y)$ are defined by 
\[ \left[\begin{array}{c} u \\ v \end{array} \right]=T\left[\begin{array}{c} x \\ y \end{array} \right]=\left[\begin{array}{c} qx \\ -px-\omega y \end{array} \right]\,.\]
Thus,
\begin{eqnarray*}
    \left[\begin{array}{c} x' \\ y' \end{array} \right] &=& 
    T^{-1}\left[\begin{array}{c} u' \\ v' \end{array} \right]=-\frac{1}{\omega q} \left[\begin{array}{c} -\omega a_1 qx(-px-\omega y)+ \ldots \\ pa_1 qx(px-\omega y)+q^2a_2 x(-px-\omega y)+ \ldots \end{array} \right] \\
    &=& -\frac{1}{\omega} \left[\begin{array}{c} (\omega a_1 p) x^2 +(\omega^2 a_1) xy + \ldots \\ -p(a_1p+a_2 q) x^2 -\omega (a_1p+a_2 q) xy + \ldots \end{array} \right] = \left[\begin{array}{c} f(x,y) \\ g(x,y) \end{array} \right]\,,
\end{eqnarray*} 
where linear and constant terms are not written explicitly, since they will not be needed in applying the Hopf theorem.
Then
\[ f_{xx}=-2a_1 p, \quad f_{xy}=-\omega a_1,\quad g_{xx}=2\frac{p}{\omega}(a_1p+a_2q), \quad g_{xy}=(a_1p+a_2q),\] 
so, following the notation in \cite[Theorem 3.4.2]{gh1983},
\begin{eqnarray*} 
a&=&\frac{1}{16\omega}[f_{xx}f_{xy}-g_{xx}g_{xy}-f_{xx}g_{xx}]\\
&=&\frac{1}{16\omega}\left[ 2\omega a_1^2 p -2\frac{p}{\omega}(a_1 p+a_2 q)^2+4\frac{a_1 p^2}{\omega}(a_1p+a_2q)]  \right]\\
&=& \frac{p}{8\omega^2}\left[  \omega^2a_1^2+(a_1p+a_2q)[-(a_1p+a_2q)+2a_1p] \right] \\
&=& \frac{p}{8\omega^2}\left[  \omega^2a_1^2 +a_1^2p^2 - a_2^2q^2 \right]\,.
\end{eqnarray*}
Now, note that $\omega^2=ps-rq=-p^2-rq$, so 
\[ a=\frac{pq}{8(p^2+rq)}(a_1^2r+a_2^2q)\]
and since $p^2+rq<0$, the Hopf bifurcation is supercritical ($a<0$, according to \cite[Theorem 3.4.2]{gh1983}) when
\begin{equation} pq(ra_1^2+qa_2^2)>0 \,.
\label{eq:supercritical}
\end{equation}
This can be expressed almost entirely in terms of the original parameters, as follows.
At a fixed point, by~\eqref{eq:fixedpoint}, 
\begin{equation*} p = a_1v^*+b_1=-a_1\left(\frac{b_1u^*+d_1}{a_1u^*+c_1}\right)+b_1 = \frac{b_1c_1-a_1d_1}{a_1u^*+c_1} \,,
\end{equation*}
so 
\begin{equation} p = \frac{b_1c_1-a_1d_1}{q}
\label{eq:p}
\end{equation}
and, similarly,
\begin{equation*} r = a_2v^*+b_2=-a_2\left(\frac{b_2u^*+d_2}{a_2u^*+c_2}\right)+b_2 = \frac{b_2c_2-a_2d_2}{a_2u^*+c_2} \,,
\end{equation*}
so
\begin{equation}
    r = \frac{b_2c_2-a_2d_2}{s}\,.
\label{eq:r}
\end{equation}
Thus, $pq=b_1c_1-a_1d_1$ and since $s=-p$, the supercriticality condition becomes
\[ \frac{(b_1c_1-a_1d_1)}{p}\left[ a_1^2(a_2d_2-b_2c_2)+a_2^2(b_1c_1-a_1d_1)  \right] >0 \,. \]
\end{proof}

\subsection{Hopf bifurcations in ambiguous cases}
\label{ss_hopfexamples}
As anticipated in the introduction, (a3) and (a4) of Theorem 6.2 in \cite{gh2015} summarize the cases in which the behavior of the solution after entering a codimension-2 domain cannot be determined by the singular perturbation approach. The aim of this section is to simplify the Hopf bifurcation condition obtained above in the situation described by (a4). 

\subsubsection{General conditions for a supercritical Hopf bifurcation in (a4)}
\label{ss_a4Hopf}
We assume, without loss of generality --- as done in \cite{gh2015} --- that $\Sigma_{\alpha}\cap\Sigma_{\beta}$ is entered through a codimension-1 sliding mode along $\Sigma_{\alpha}$ from the region $y^{(2)}<\theta_2$. The other main condition which defines the cases analyzed in Theorem 6.2 of \cite{gh2015} is that the nullcline $v_u(u)$ \emph{turns right}; that is, given $u_0$ such that $v_u(u_0)=-1$, the branch of the nullcline starting at $u_0$ is confined in the region $u\geq u_0$. Let $(u^*,v^*)$ be the fixed point closed to $(u_0,-1)$, which, as proved in Theorem 6.2 of \cite{gh2015}, is the only one which might be stable. The (a4) case requires that the branch of $v_v(u)$ containing $(u^*,v^*)$ ends up on the right edge of $[-1,1]^2$ and that $f_1^{++}(\theta_1,\theta_2)>0$. The conditions above imply that the slopes of the nullclines $v_u(u),v_v(u)$ have to satisfy
$\frac{dv_u}{du}|_{u=u^*}=\frac{a_1d_1-b_1c_1}{q^2} >0$ and $\frac{dv_v}{du}|_{u=u^*}=\frac{a_2d_2-b_2c_2}{s^2} >0$, as well as $\frac{dv_u}{du}|_{u=u^*}<\frac{dv_v}{du}|_{u=u^*}$, so $\frac{a_1d_1-b_1c_1}{q^2} < \frac {a_2d_2-b_2c_2}{s^2}$.
 All three of these conditions can be summarized as
\[ \frac{a_2d_2-b_2c_2}{s^2} > \frac{a_1d_1-b_1c_1}{q^2}  > 0\,, \]
{\em i.e.},
\[ \frac{a_2d_2-b_2c_2}{(a_2u^*+c_2)^2}>\frac{a_1d_1-b_1c_1}{(a_1 u^* + c_1)^2} >0\,. \]
At a fixed point, by~\eqref{eq:p} and~\eqref{eq:r}, this means 
\[ \frac rs < \frac pq <0\,.\]

In order for a Hopf bifurcation to be possible, we must have $\text{Tr}^2(J^*)-4\text{det}(J^*)<0$, which is $(p+s)^2-4(ps-qr)<0$ or, equivalently, $(p-s)^2+4qr<0$, so it is necessary that $qr<0$. Also, at a Hopf bifurcation point, $p=-s$. And the (a4) fixed point condition above implies that $r$ and $s$ have opposite sign, as do $p$ and $q$. Thus,
\[ \text{sign}(p)=\text{sign}(r)=-\text{sign}(q)=-\text{sign}(s).\]
Thus, $pq<0$ and the supercriticality condition~\eqref{eq:supercritical} reduces to 
\[ ra_1^2+qa_2^2<0,\]
which can be written, by~\eqref{eq:p} and \eqref{eq:r} again, as
\[ \frac{a_1^2(a_2d_2-b_2c_2)-a_2^2(a_1d_1-b_1c_1)}{p}<0.\]
Thus, if $p=a_1v^*+b_1>0$, the supercriticality condition is
\[ a_1^2(a_2d_2-b_2c_2)-a_2^2(a_1d_1-b_1c_1)<0, \]
and if $p<0$, the supercriticality condition is
\[ a_1^2(a_2d_2-b_2c_2)-a_2^2(a_1d_1-b_1c_1)>0. \]
In the following subsections, we will use the results above in order to analyze two different systems whose hidden dynamics is described by means of either ramp or Hill functions which are, in turn, defined by a parameter $\kappa$ (see Remark \ref{kepsilon}). Our main goal is to detect the Hopf bifurcations with respect to $\kappa$ and analyze their nature.

\subsubsection{Supercritical Hopf bifurcation with both ramp and Hill functions}
Consider the system

\begin{equation}
\begin{split}
    u'&=uv-u+v \\
    v'&=\kappa(a_2 uv-2u+v),
\end{split}
 \label{eq:uvexample1}
\end{equation}
which can be interpreted as the hidden dynamics of a discontinuous system in a codimension-2 domain, regularized using ramp functions having slopes $\frac{1}{\varepsilon}$ in the first variable and $\frac{\kappa}{\varepsilon}$ in the second variable.
If $1<a_2<\sqrt{2}$, then \eqref{eq:uvexample1} is an instance of an (a4) system~\cite{gh2015} with $f^{-+}$ pointing upwards and there is a supercritical Hopf bifurcation at $\kappa=1$. When $\kappa<1$, the fixed point at $(0,0)$ is stable. When $\kappa>1$, the fixed point is unstable and there is a stable limit cycle around it. This can be checked by the conditions of subsection~\ref{ss_a4Hopf} above as follows:
\[a_1=1, b_1=-1, c_1=1, d_1=0, b_2=-2, c_2=1, d_2=0.\]
The Jacobian matrix at the origin is
\[ J^*=\left[\begin{array}{cc}-1 & 1\\ -2\kappa & \kappa \end{array}\right]\,, \]
so $p=-1, q=1, r=-2\kappa, s=\kappa$. 
At $\kappa=1$, $p=-s$, so $\text{Tr}(J^*)=0$ and $\text{det}(J^*)=ps-qr=1$, so we have a Hopf bifurcation. The (a4) condition is satisfied, since $\frac rs=-2<-1=\frac pq < 0$. The supercriticality condition is satisfied since $p<0$ and
$a_1^2(a_2d_2-b_2c_2)-a_2^2(a_1d_1-b_1c_1)=2-a_2^2>0$
for $a_2<\sqrt{2}$.

Numerical simulation from the entry point $(u_0,v_0)=\left(-\frac 12, -1\right)$ to the box $[-1,1]^2$, with $a_2=\frac{11}{10}$ and $\kappa=1.001$, shows an asymptotic approach to the limit cycle.

Numerical study of \eqref{eq:uvexample1} translated to the Hill function context, so that the right hand side of each equation is premultiplied by $\frac 14 (1-u^2)$ and $\frac 14 (1-v^2)$, respectively, behaves similarly: when $\kappa$ is taken slightly past the bifurcation point ($\kappa=1.01$, for example) the solution entering from the black wall below converges to a stable limit cycle around the fixed point at $(0,0)$. The supercriticality can be confirmed analytically as follows. The new equations read
\begin{equation}
\begin{split}
    u'&=\frac14(1-u^2)\left(uv-u+v\right)=-\frac{1}{4}u^3v+\frac14u^3-\frac14u^2v+\frac{1}{4}uv-\frac14u+\frac14v  \\
    v'&=\frac14\kappa(1-v^2)\left(a_2 uv-2u+v\right) =\kappa\left(-\frac{1}{4}a_2uv^3+\frac12uv^2-\frac14v^3+\frac{1}{4}a_2uv-\frac12u+\frac14v\right)\,.
\end{split}
 \label{eq:uvexample1_hill}
\end{equation}
Note that the system obtained is not multilinear; therefore we cannot directly apply the results obtained at the beginning of the present section.
The Jacobian matrix at the origin defines $p=-\frac14, q=\frac14, r=-\frac12, s=\frac14$ for the bifurcation value $\kappa=1$. Thus,
\[ T=\left[\begin{array}{cc} \frac14 & 0 \\ \frac14 & -\frac{1}{4}\end{array} \right]\,,\quad\quad T^{-1}=\left[\begin{array}{cc} 4 & 0 \\ 4 & -4\end{array} \right]\,, \]
and
\begin{eqnarray*}
    \left[\begin{array}{c} x' \\ y' \end{array} \right] &=& 
    \left[\begin{array}{c} u^3-u^2v+uv+ \ldots \\ u^3+v^3-u^2v-2uv^2+(1-a_2)uv+ \ldots \end{array} \right],
\end{eqnarray*} 
where fourth-order terms (as well as linear and constant ones) are not written explicitly since they do not affect the result, given that $(u^*,v^*)=(0,0)$. We get
\begin{eqnarray*}
f &=& \left(\frac14x\right)^3-\left(\frac14x\right)^2\left(\frac14x-\frac{1}{4}y\right)+\left(\frac14x\right)\left(\frac14x-\frac{1}{4}y\right)\\
&=& \frac{1}{64}x^2y+\frac{1}{16}x^2-\frac{1}{16}xy,
\end{eqnarray*} 
so that $f_x=\frac{1}{32}xy+\frac{1}{8}x-\frac{1}{16}y$, $f_{xx}=\frac{1}{32}y+\frac{1}{8}$, $f_{xy}=\frac{1}{32}x-\frac{1}{16}$, and $f_{yy}=f_{xyy}=f_{xxx}=0$. As for $g$,
\begin{eqnarray*}
g &=& \left(\frac14x\right)^3+\left(\frac{1}{64}x^3-\frac{1}{64}y^3-\frac{3}{64}x^2y+\frac{3}{64}xy^2\right)\\
&-&\left(\frac14x\right)^2\left(\frac14x-\frac{1}{4}y\right)-2\left(\frac14x\right)\left(\frac{1}{16}x^2+\frac{1}{16}y^2-\frac{1}{8}xy\right)\\
&+&\frac{1-a_2}{4}x\left(\frac14x-\frac{1}{4}y\right)\\
&=& -\frac{1}{64}x^3-\frac{1}{64}y^3+\frac{1}{32}x^2y+\frac{1}{64}xy^2+\frac{1-a_2}{16}x^2+\frac{a_2-1}{16}xy,
\end{eqnarray*} 
so that $g_x=-\frac{3}{64}x^2+\frac{1}{16}xy+\frac{1}{64}y^2+\frac{1-a_2}{8}x+\frac{a_2-1}{16}y$, $g_y=-\frac{3}{64}y^2+\frac{1}{32}x^2+\frac{1}{32}xy+\frac{a_2-1}{16}x$, $g_{xx}=-\frac{3}{32}x+\frac{1}{16}y+\frac{1-a_2}{8}$, $g_{xy}=\frac{1}{16}x+\frac{1}{32}y+\frac{a_2-1}{16}$, $g_{yy}=-\frac{3}{32}y+\frac{1}{32}x$, $g_{xxy}=\frac{1}{16}$, and $g_{yyy}=-\frac{3}{32}$. Eventually we get
\begin{eqnarray*}
a &=& \frac{1}{16}(g_{xxy}+g_{yyy})+\frac{1}{16\omega}(f_{xx}f_{xy}-g_{xx}g_{xy}-f_{xx}g_{xx})=\\
&=&\frac{1}{16}\cdot\left(-\frac{1}{32}\right)+\frac{1}{4}\left(-\frac{1}{8}\cdot\frac{1}{16}+\frac{(a_2-1)^2}{128}+\frac{a_2-1}{64}\right)=\frac{a_2^2-2}{512}<0
\end{eqnarray*} 
for $a_2<\sqrt{2}$. 

Going back to numerical simulation of the system with ramp functions, {\em i.e.},~\eqref{eq:uvexample1}, with $a_2=\frac{11}{10}$ and $\kappa=1.001$ the solution from the black wall below converges to a stable limit cycle. As $\kappa$ is increased, the amplitude of the limit cycle grows, and already for $\kappa=1.05$, while the limit cycle itself is still in the $[-1,1]^2$ box, the trajectory entering from the black wall at $u_0=-\frac 12$, $v_0=-1$ exits the box slightly through the lower edge, and then reenters. Thus, technically, the trajectory reenters the black wall momentarily, but only on the $\varepsilon$ scale and while $v$ is also still evolving on the $\varepsilon$ scale, so it is able to reenter the box. 

As $\kappa$ is increased further to, e.g., $\kappa=1.1$, the limit cycle itself extends outside the box, effectively passing through the black wall to the right, then the $\mathcal{R}^{+-}$ classical domain, and then the black wall below, before reentering $[-1,1]^2$. Again, all of this happens on the $\varepsilon$ scale. 

At $\kappa=1.2$, however, the limit cycle has ceased to exist, and the solution exits the box $[-1,1]^2$ into the $\mathcal{R}^{++}$ domain, where a classical solution exists, so it goes away from the threshold intersection in the large space.

If we do the same exercise in the Hill function context, {\em i.e.}, with reference to \eqref{eq:uvexample1_hill}, the behavior is somewhat different. It is clear that the factors $(1-u^2)$ and $(1-v^2)$ that multiply the right hand sides imply that $u=\pm 1$ are nullclines for $u$, and $v=\pm 1$ are nullclines for $v$. So it is not possible in this context for a limit cycle of the fast variables to go outside the box $[-1,1]$ and the only way for any trajectory to exit the box is through a fixed point on the boundary. Of course, it is only fixed for the fast variables, not for the original variables in the large space.

In our example, the stable limit cycle persists for $\kappa$ much larger than in the ramp function case. For $\kappa=1.66$ we still see the limit cycle numerically, and converge to it from the entry point at $u_0=-\frac 12, v_0=-1$, though it is now large and approaches very close to the right and lower boundaries of the box $[-1,1]^2$.  At $\kappa=1.68$, the trajectory from the entry point leaves through the point $(1,1)$, and goes into the $\mathcal{R}^{++}$ domain as a classical solution. See Figure \ref{fig:heteroclinic}, which shows how the qualitative behavior of the solution varies with $\kappa$ and compares the solution at the bifurcation value with that of the corresponding singularly perturbed problem. In the latter, the boundaries represented correspond to $y_i=\theta_i\pm\varepsilon$, $i=1,2$.

\begin{figure}[t!]
    \begin{center}
     \subfigure[]
 {    \includegraphics[width=3.6cm]{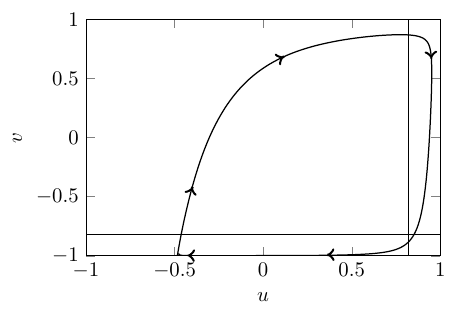}  
 }\hspace{-0.3cm}
 \subfigure[]{
 \includegraphics[width=3.6cm]{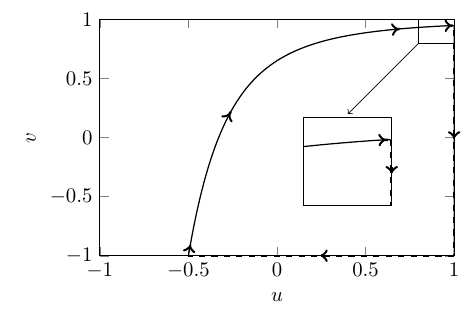}}\hspace{-0.3cm}
 \subfigure[]{
 \includegraphics[width=3.6cm]{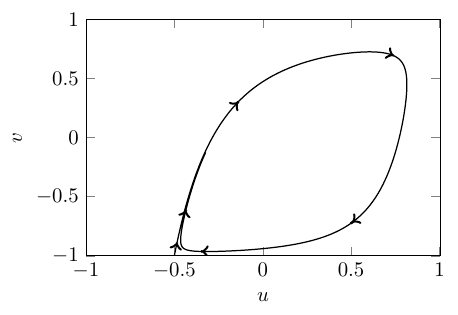}}\hspace{-0.3cm}
 \subfigure[]{
    \includegraphics[width=3.6cm]{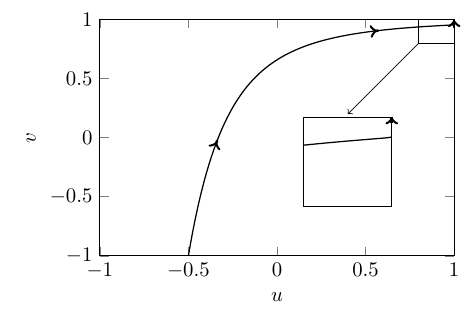}}

        \subfigure
 {    \includegraphics[width=3.6cm]{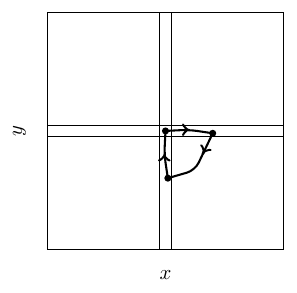}  
 }\hspace{-0.3cm}
 \subfigure{
 \includegraphics[width=3.6cm]{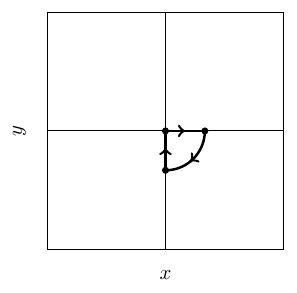}}\hspace{-0.3cm}
 \subfigure{
 \includegraphics[width=3.6cm]{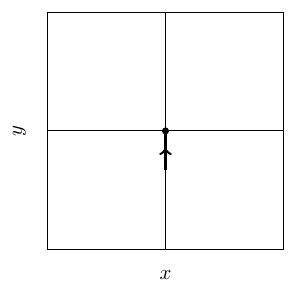}}\hspace{-0.3cm}
 \subfigure{
    \includegraphics[width=3.6cm]{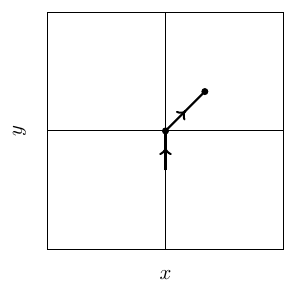}}
    \caption{Equations~\eqref{eq:uvexample1_hill} with $a_2=\frac{11}{10}$, $u(0)=-0.5, v(0)=-1$ and (a) $\varepsilon=10^{-2}$, $\kappa=1.66$, (b) $\varepsilon=0$, $\kappa=1.66$, (c) $\kappa=1.25$, (d) $\kappa=1.68$, simulated in the fast variables (top) and sketched in the original variables (bottom). This shows the sensitivity of the final trajectory with respect to the choice of $\kappa$. Pictures (a) and (b) describe the non-generic case in which the trajectory in the $[-1,1]^2$ box hits exactly the saddle point on the right edge. Pictures (b) and (d) include a zoom of the upper-right corner.
    \label{fig:heteroclinic}}
    \end{center}
\end{figure}
The bifurcation between these two behaviours appears to be a heteroclinic bifurcation. This happens when the limit cycle approaches, and asymptotically touches, the fixed point where the $v$-nullcline intersects the right edge at $u=1$. This point, $(u^*,v^*)=(1,\frac{2}{a_2+1})$, is unstable in the $v$ direction; the flow on the edge is towards either the corner at $(1,1)$ or at $(1,-1)$. The Jacobian matrix at this point is
\begin{eqnarray*}
    J(u^*,v^*)&=&\frac 14 \left[\begin{array}{cc}(-2u^*)(u^*v^*-u^*+v^*) & 0\\ \kappa(1-(v^*)^2)(a_2 v^*-2) & \kappa(1-(v^*)^2)(a_2u^*+1) \end{array}\right] \\
&=&\left[\begin{array}{cc}-2\frac{3-a_2}{a_2+1} & 0\\ -2\frac{(a_2+1)^2-4}{(a_2+1)^3}\kappa & \frac{(a_2+1)^2-4}{a_2+1}\kappa \end{array}\right].
\end{eqnarray*} 
This has a positive eigenvalue $\frac{(a_2+1)^2-4}{a_2+1}\kappa$ corresponding to the vertical eigenvector, which is actually the unstable manifold of the saddle point. The other eigenvalue is $-2\frac{3-a_2}{a_2+1}$, and the stable manifold enters the point on the eigenvector $((a_2+1)^2(((a_2+1)^2-4)\kappa+2(3-a_2)),\, 2((a_2+1)^2-4)\kappa)$. 

In order to demonstrate that there is a bifurcation, first consider the rectangle 
\[ R=\left[-\frac 12, 1\right]\times \left[-1, \frac{2}{a_2+1}\right]\,. \]
The flow on the left boundary of $R$ at $u=-\frac 12$ is to the right, since $u'=\frac{3}{32}(v+1)$ there. The lower and right edges of $R$ are boundaries of $[-1,1]^2$ and therefore are solution trajectories themselves. Thus, the only way to exit $R$ (potentially) is through the upper edge at $v=\frac{2}{a_2+1}$.

For $\kappa=1$, where the limit cycle first appears, the flow from the entry point, numerically, does not exit $R$ at the top edge but crosses the $v$ nullcline at a lower point, and turns downwards. An analytic proof of this claim is given in Appendix \ref{s_appendix}. From the entry point, $u'>0$ and $v'>0$ on the first part of the trajectory, which cannot cross the $u$ nullcline below $v=0$, since the flow across that curve must be from below to above. It must continue moving upward and to the right until either it exits $R$ through $v=\frac{2}{a_2+1}$ or it crosses the $v$ nullcline and goes downwards. As $\kappa$ increases, it is clear that the vertical component of the vector field increases everywhere (except where it is $0$), and so the initial part of the trajectory from the entry point, until it crosses the $v$ nullcline (should that occur), must move upwards with steeper slope. As $\kappa\to\infty$, the solution approaches the vertical line at $u=-\frac 12$, and so exits $R$ close to the point $(-\frac 12, \frac{2}{a_2+1})$.

Now consider a $\kappa>1$ small enough that the trajectory from the entry point does not leave $R$, but crosses the $v$ nullcline. This trajectory must then go downward and to the right ($u'>0, v'<0$) until the $u$ nullcline is crossed. Then it continues downward and to the left ($u'<0, v'<0$) until it crosses the lower part of the $v$ nullcline, and from there, upward and to the left ($u'<0, v'>0$) until it crosses the lower part of the $u$ nullcline. Call this point $(u_2,v_2)$. Let the point $(u_1,v_1)$ be the point where the initial segment of the trajectory first reached $v=v_2$. Thus, $v_1=v_2$ but $u_1<u_2$. 

Consider the region ${\cal D}$ bounded by the trajectory from $(u_1,v_1)$ to $(u_2,v_2)$ and the straight line segment between those two points. ${\cal D}$ is an invariant region, since the only part of the boundary that is not a trajectory itself is the straight line segment between $(u_1,v_1)$ and $(u_2,v_2)$, and the flow there is upwards ($v'>0$). Inside ${\cal D}$ there is only one fixed point, the unstable spiral point at $(0,0)$. Since that point is a repellor, the continued trajectory from $(u_2,v_2)$, which is confined to ${\cal D}$, must converge to a limit cycle by the Poincar\'e--Bendixson theorem; see Figure \ref{fig:P-B}.

\begin{figure}
    \begin{center}
        \includegraphics[width=8cm]{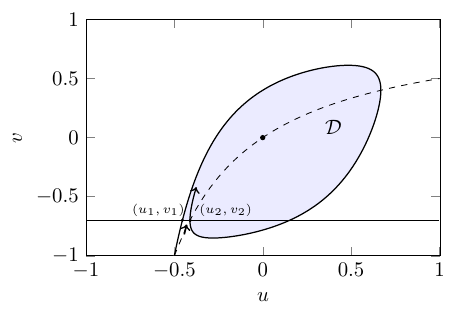}    
    \caption{Invariant region $\mathcal{D}$ on which the Poincar\'e-Bendixson Theorem is applied.}
    \label{fig:P-B}
    \end{center}
\end{figure}

On the other hand, a trajectory that exits $R$ on $v=\frac{2}{a_2+1}$ cannot cross the $v$ nullcline, since $v'>0$ above $v=\frac{2}{a_2+1}$, so the solution must go to $(1,1)$ asymptotically. (Since this occurs on the fast ($\tau$) time scale, its asymptotic approach to $(1,1)$ implies exit from the threshold intersection into the $\mathcal{R}^{++}$ domain, and the solution continues there in slow time.) At $\kappa=1$, the initial trajectory hits the $v$ nullcline. As $\kappa$ increases, the point at which that crossing occurs moves up along the nullcline. Eventually (certainly as $\kappa\to\infty$), the initial trajectory exits on the top edge of $R$. So there must be a bifurcation point $\kappa=\kappa^*$ such that the initial trajectory hits the intersection of the $v$ nullcline and the top edge of $R$, which is the saddle point $\left(1,\frac{2}{a_2+1}\right)$. 

This is enough to show that there is a bifurcation in behaviour at a parameter value $\kappa=\kappa^*$, below which the trajectory from the entry point converges to a limit cycle, and above which the trajectory from the entry point exits the $[-1,1]^2$ box at $(1,1)$ and goes to a classical solution in the $\mathcal{R}^{++}$ domain.

Although not crucial to our main argument, it would be nice to demonstrate that this bifurcation is, in fact, a heteroclinic bifurcation, in which the limit cycle collides with the saddle point and disappears, but for this more is needed.
The existence of a parameter value $\kappa^*$ such that the initial trajectory coincides with the stable manifold ($W^s$) of the saddle point at $\left(1,\frac{2}{a_2+1}\right)$ is one of the ingredients. The unstable manifold ($W^u$) of the saddle point is the $u=1$ edge of the box $[-1,1]^2$. Guckenheimer and Holmes give a version of the homoclinic bifurcation theorem~\cite[Theorem 6.1.1, p.292]{gh1983}, for which one needs a saddle loop, which means a homoclinic orbit, so that $W^s$ coincides with $W^u$ on the loop. Here, we have not a homoclinic orbit, but a heteroclinic cycle, including $W^u$, the line from $\left(1,\frac{2}{a_2+1}\right)$ to $(1,-1)$, and part of the lower edge of the $[-1,1]^2$ box, from $(1,-1)$ to $(u_0,-1)$, the entry point. So the cycle contains three saddle points. Nearby trajectories inside the box of course follow this heteroclinic cycle closely.

We also need a transversal, $M$, which we can take here to be the line $u=0$, which passes through the unstable spiral point at $(0,0)$, around which any limit cycle must go. The point $S(\kappa)$ is the intersection of the stable manifold of $(1,\frac{2}{a_2+1})$ with $M$, and the point $U(\kappa)$ is the intersection of the trajectory from the entry point with $M$; see Figure \ref{fig:homoclinic}.

\begin{figure}
    \begin{center}
        \includegraphics[width=8cm]{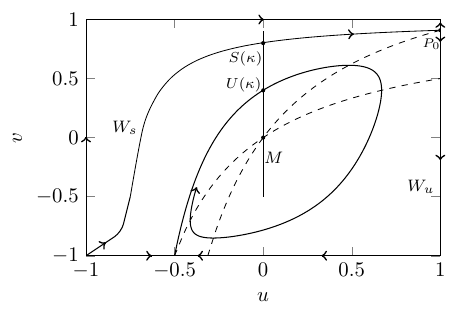}    
    \caption{Representation of the homoclinic bifurcation theorem. Here, in the limiting system ($\varepsilon\to 0$), at the bifurcation value, the trajectory is actually heteroclinic.
    \label{fig:homoclinic}}
    \end{center}
\end{figure}
It is clear that the $v$ coordinate of $S(\kappa)$, which we denote by $S_v(\kappa)$, is a decreasing function of $\kappa$, since increasing $\kappa$ increases the vertical distance traversed by the trajectory on $W^s$ between $u=0$ and $u=1$. Thus, $\frac{d}{d\kappa}S_v(\kappa)<0$. Similarly, $\frac{d}{d\kappa}U_v(\kappa)>0$, where $U_v$ is the $v$ component of $U$, since as $\kappa$ increases, the vertical distance traversed by the trajectory from $u=u_0$ to $u=0$ increases. Clearly, beyond the bifurcation point, for $\kappa>\kappa^*$, $U_v(\kappa)>S_v(\kappa)$, while before the bifurcation point, $U_v(\kappa)<S_v(\kappa)$, and at $\kappa=\kappa^*$, $U_v(\kappa^*)=S_v(\kappa^*)$. The theorem requires that $\left.\frac{d}{d\kappa}(U_v(\kappa)-S_v(\kappa))\right|_{\kappa=\kappa^*}\ne 0$. 

We cannot use the homoclinic bifurcation theorem, because we actually have a heteroclinic cycle. The above construction of $M$, $S(\kappa)$, and $U(\kappa)$ shows that a bifurcation takes place. We do not have a rigorous proof, however, to conclude that the stable limit cycle disappears, though this seems inevitable and numerical simulation supports the claim. It may also be the case that for any $\varepsilon>0$, we do have a homoclinic orbit at the bifurcation point. In that case, we still have a saddle point in the white wall corresponding to the $u=1$ edge of the $[-1,1]^2$ box, but no saddle points corresponding to $(1,-1)$ or $(u_0,-1)$. In fact, even in the limit $\varepsilon\to 0$, we have a homoclinic orbit in the large phase space, but the homoclinic bifurcation theorem cannot be applied directly to discontinuous systems. Thus, we have proven that a bifurcation occurs on one side of which the trajectory from the entry point remains confined to the switching domain, and on the other side of which it escapes. We tentatively identify this as a heteroclinic bifurcation in the limit system.

Note that these Hill function equations~\eqref{eq:uvexample1_hill} derive from macroscopic equations in the original slow variables:
\begin{eqnarray*}
    \dot{x}&=&(uv-u+v+\gamma_1\theta_1)-\gamma_1 x \\
    \dot{y}&=&\left(a_2 uv-2u+v+\gamma_2\theta_2\right) -\gamma_2 y\,,
\end{eqnarray*}
where $\gamma_1\theta_1\ge 3$ for a gene network, since $x$ should be confined to nonnegative values, and similarly $\gamma_2\theta_2\ge 3+a_2$. If we take $\theta_1=\theta_2=1$, then we may have $\gamma_1=3$ and $\gamma_2=3+a_2$, for example. Then the unstable fixed point (singular stationary point) in the white wall is at $x=\frac{3-a_2}{\gamma_1(a_2+1)}+\theta_1=\frac{82}{63}$ for $a_2=\frac{11}{10}$, and $y=\theta_2=1$. This comes from the fact that $v=\frac{2}{a_2+1}$ is its equilibrium value in the codimension-1 region corresponding to the white wall, and then the slow flow in $x$ is governed by its differential equation with $u=1$ and $v=\frac{2}{a_2+1}$:
\[ \dot{x}=\left(\frac{2}{a_2+1}(1)-1+\frac{2}{a_2+1}+\gamma_1\theta_1\right)-\gamma_1 x =\left(\frac{3-a_2}{a_2+1}+\gamma_1\theta_1\right)-\gamma_1 x\,, \]
so $x=\frac{1}{\gamma_1}\left(\frac{3-a_2}{a_2+1}+\gamma_1\theta_1\right)$ is its equilibrium value.

The saddle loop trajectory at the bifurcation point $\kappa=\kappa^*$ in the hidden dynamics corresponds in the large phase space to a solution that goes from the threshold intersection to the unstable fixed point in the white wall, then back through the $\mathcal{R}^{+-}$ domain to the black wall, and up again to the threshold intersection; see Figure \ref{fig:heteroclinic}. It should be emphasized that this macroscopic trajectory occurs only at the bifurcation point. Before the bifurcation, the solution remains in the threshold intersection, while afterwards, it exits to the $\mathcal{R}^{++}$ domain. If one considers $\varepsilon>0$ but small, then this picture should be preserved (whether in $x,y$ coordinates or $u,v$ coordinates): at the perturbed bifurcation point, the flow that enters the threshold intersection region (now nonzero in size) lies on the stable manifold of the perturbed saddle point in the white wall region, and the stable manifold of this saddle point coincides with the unstable manifold in a homoclinic loop. It is no longer heteroclinic, because there are no longer saddle points at $(1,-1)$ and $(u_0,-1)$. The flow now passes through the corresponding regions, which now have nonzero size. Thus, the homoclinic bifurcation theorem could be applied in this context, and then the conclusion drawn in the limit as $\varepsilon\to 0$. However, it would have to be demonstrated that the picture for $\varepsilon>0$ is indeed qualitatively the same.

While the microscopic ($\varepsilon$ scale) behavior of the Hill function system and the ramp function system differ, the large-scale qualitative behavior is the same, though the bifurcations in behavior (from convergence to limit cycle in the fast variables, to exit into the $\mathcal{R}^{++}$ domain) occur at different values of parameters. 

\subsubsection{Hopf bifurcation changing nature according to switching function}

Consider the system

\begin{equation}
\begin{split}
    u'&=\frac{11}{2}uv-u+5v  \\
    v'&=\kappa\left(a_2 uv-\frac{1}{2}u+v\right),
\end{split}
 \label{eq:uvexample2}
\end{equation}
which can be interpreted similarly to \eqref{eq:uvexample1} in the previous example.
If $\frac{3}{2}<a_2<\frac{11\sqrt{10}}{20}$, then \eqref{eq:uvexample2} is an instance of an (a4) system \cite{gh2015} with $f^{-+}$ pointing downwards and all other components of the vector fields having the same sign as in \eqref{eq:uvexample1}. Again, there is a supercritical Hopf bifurcation at $\kappa=1$, and the fixed point at $(0,0)$ is stable when $\kappa<1$ and unstable when $\kappa>1$, with a stable limit cycle around it. This can be checked by the conditions of subsection~\ref{ss_a4Hopf} as follows:
\[a_1=\frac{11}{2}, b_1=-1, c_1=5, d_1=0, b_2=-\frac12, c_2=1, d_2=0.\]
The Jacobian matrix at the origin is
\[ J^*=\left[\begin{array}{cc}-1 & 5\\ -\frac12\kappa & \kappa \end{array}\right]\,, \]
so $p=-1, q=5, r=-\frac12\kappa, s=\kappa$. 
At $\kappa=1$, $p=-s$, so $\text{Tr}(J^*)=0$ and $\text{det}(J^*)=ps-qr=\frac32$, so we have a Hopf bifurcation. The (a4) condition is satisfied, since $\frac rs=-\frac12<-\frac15=\frac pq < 0$. The supercriticality condition is satisfied since $p<0$ and
$a_1^2(a_2d_2-b_2c_2)-a_2^2(a_1d_1-b_1c_1)=\frac{121}{8}-5a_2^2>0$ for $a_2<\frac{11\sqrt{10}}{20}$.

Numerical simulation from the entry point $(u_0,v_0)=\left(-\frac {10}{13}, -1\right)$ to the box $[-1,1]^2$,  with $a_2=\frac{8}{5}$ and $\kappa=1.001$, shows an asymptotic approach to the limit cycle; however, it exits the box at the right edge first and then reenters.

If we translate this system to the Hill function context, while keeping variables in $[-1,1]$, we get
\begin{equation}
\begin{split}
    u'&=\frac14(1-u^2)\left(\frac{11}{2}uv-u+5v\right)=-\frac{11}{8}u^3v+\frac14u^3-\frac54u^2v+\frac{11}{8}uv-\frac14u+\frac54v  \\
    v'&=\frac14\kappa(1-v^2)\left(\frac85 uv-\frac{1}{2}u+v\right) =\kappa\left(-\frac25uv^3+\frac18uv^2-\frac14v^3+\frac25uv-\frac18u+\frac14v\right)\,.
\end{split}
 \label{eq:uvexample2_hill}
\end{equation}
The Jacobian matrix at the origin defines $p=-\frac14, q=\frac54, r=-\frac18, s=\frac14$ for the bifurcation value $\kappa=1$. Thus,
\[ T=\left[\begin{array}{cc} \frac54 & 0 \\ \frac14 & -\frac{\sqrt{6}}{8}\end{array} \right]\,,\quad\quad T^{-1}=\left[\begin{array}{cc} \frac45 & 0 \\ \frac{4\sqrt{6}}{15} & -\frac{4\sqrt{6}}{3}\end{array} \right]\,, \]
and
\begin{eqnarray*}
    \left[\begin{array}{c} x' \\ y' \end{array} \right] &=& 
    \left[\begin{array}{c} -\frac{11}{10}u^3v+\frac15u^3-u^2v+\frac{11}{10}uv+ \ldots \\ \frac{\sqrt{6}}{15}u^3+\frac{\sqrt{6}}{3}v^3-\frac{\sqrt{6}}{3}u^2v-\frac{\sqrt{6}}{6}uv^2-\frac{\sqrt{6}}{6}uv+ \ldots \end{array} \right],
\end{eqnarray*} 
where fourth-order terms (as well as linear and constant ones) are not written explicitly since they do not affect the result, given that $(u^*,v^*)=(0,0)$. We get
\begin{eqnarray*}
f &=& \frac15\left(\frac54x\right)^3-\left(\frac54x\right)^2\left(\frac14x-\frac{\sqrt{6}}{8}y\right)+\frac{11}{10}\left(\frac54x\right)\left(\frac14x-\frac{\sqrt{6}}{8}y\right)\\
&=& \frac{25\sqrt{6}}{128}x^2y+\frac{11}{32}x^2-\frac{11\sqrt{6}}{64}xy,
\end{eqnarray*} 
so that $f_x=\frac{25\sqrt{6}}{64}xy+\frac{11}{16}x-\frac{11\sqrt{6}}{64}y$, $f_{xx}=\frac{25\sqrt{6}}{64}y+\frac{11}{16}$, $f_{xy}=\frac{25\sqrt{6}}{64}x-\frac{11\sqrt{6}}{64}$, and $f_{yy}=f_{xyy}=f_{xxx}=0$. As for $g$,
\begin{eqnarray*}
g &=& \frac{\sqrt{6}}{15}\left(\frac54x\right)^3+\frac{\sqrt{6}}{3}\left(\frac{1}{64}x^3-\frac{3\sqrt{6}}{256}y^3-\frac{3\sqrt{6}}{128}x^2y+\frac{9}{128}xy^2\right)\\
&-&\frac{\sqrt{6}}{3}\left(\frac54x\right)^2\left(\frac14x-\frac{\sqrt{6}}{8}y\right)-\frac{\sqrt{6}}{6}\left(\frac54x\right)\left(\frac{1}{16}x^2+\frac{3}{32}y^2-\frac{\sqrt{6}}{16}xy\right)\\
&-&\frac{\sqrt{6}}{6}\left(\frac54x\right)\left(\frac14x-\frac{\sqrt{6}}{8}y\right)\\
&=& -\frac{\sqrt{6}}{128}x^3-\frac{3}{128}y^3+\frac{27}{64}x^2y+\frac{\sqrt{6}}{256}xy^2-\frac{5\sqrt{6}}{96}x^2+\frac{5}{32}xy,
\end{eqnarray*} 
so that $g_x=-\frac{3\sqrt{6}}{128}x^2+\frac{27}{32}xy+\frac{\sqrt{6}}{256}y^2-\frac{5\sqrt{6}}{48}x+\frac{5}{32}y$, $g_y=-\frac{9}{128}y^2+\frac{27}{64}x^2+\frac{\sqrt{6}}{128}xy+\frac{5}{32}x$, $g_{xx}=-\frac{3\sqrt{6}}{64}x+\frac{27}{32}y-\frac{5\sqrt{6}}{48}$, $g_{xy}=\frac{27}{32}x+\frac{\sqrt{6}}{128}y+\frac{5}{32}$, $g_{yy}=-\frac{9}{64}y+\frac{\sqrt{6}}{128}x$, $g_{xxy}=\frac{27}{32}$, and $g_{yyy}=-\frac{9}{64}$. Eventually we get
\begin{eqnarray*}
a &=& \frac{1}{16}(g_{xxy}+g_{yyy})+\frac{1}{16\omega}(f_{xx}f_{xy}-g_{xx}g_{xy}-f_{xx}g_{xx})=\\
&=&\frac{1}{16}\cdot\frac{45}{64}+\frac{1}{2}\left(-\frac{11}{16}\cdot\frac{11}{64}+\frac{5}{48}\cdot\frac{5}{32}+\frac{11}{16}\cdot\frac{5}{48}\right)=\frac{59}{2048}>0.
\end{eqnarray*}

In any case, the important fact is that in the system with ramp switching functions~\eqref{eq:uvexample2} we had a supercritical Hopf bifurcation and the appearance of a stable limit cycle, but when we replace the ramps with Hill functions~\eqref{eq:uvexample2_hill}, the Hopf bifurcation at the same bifurcation parameter value (and at the same fixed point) is subcritical, and no stable limit cycle appears. Numerical simulations strongly suggest the presence of an unstable limit cycle around the fixed point when $\kappa$ is slightly below the bifurcation values (e.g., $\kappa=0.99$).

\section{Concluding remarks}
\label{s_conclusions}
In this paper we analyzed the problem of nonuniqueness of solutions of piecewise smooth ODEs, which are interpreted as the limit of a class of regularized problems defined accordingly.
As observed in \cite{gh2022}, difficulties may arise upon approaching the intersection of two (or more) discontinuity hypersurfaces. 
In particular, in this article, we have extensively explored the consequences of the nonuniquemess in the case (a4) of the classification provided in \cite{gh2022}.
The contributions of this work are mainly twofold.

On the one hand we have shown, by means of both extensive numerical tests and analytic considerations, that the algorithm proposed in \cite{gh2022} to determine the behavior of a solution after entering a codimension-2 surface 
$\Sigma_\alpha \cap \Sigma_\beta$, relying on the underlying regularization, can only be extended to the cases labeled as ambiguous  by integrating the hidden dynamics and providing the numerical integrator by an event detection algorithm, able to recognize when the solution of the hidden dynamical system either reaches an equilibrium or a periodic orbit, or when one or both of the components diverge. We have shown that in certain cases one can obtain both a classical solution and a codimension-2 sliding mode by just changing the switching function. 
We expect that in higher codimension (see, e.g., \cite{gh2023}), the use of a numerical integrator for interpreting the hidden dynamics will be necessary in more and more situations, with the additional complication that further kinds of invariants could appear (e.g., chaotic attractors \cite{MEV2013a}).

On the other hand, we observed that a codimension-2 sliding mode can correspond to either an equilibrium or a periodic solution in the hidden dynamics. More specifically, we considered a class of switching functions defined by a parameter and commented on various kinds of bifurcation phenomena, with particular attention to Hopf bifurcations. A periodic attractor in the hidden dynamics can lead to behavior that is sensitive to initial conditions or parameters when the solution reaches a higher codimension switching domain. As a consequence, uniqueness in the limit solution can be lost.

\appendix
\section{Appendix}\label{s_appendix}
 Here, we prove the claim that the solution of system~\eqref{eq:uvexample1_hill} with $\kappa=1$ and $a_2>1$ sufficiently close to $1$ (for example, $a_2=\frac{11}{10}$), starting from the entry point at $(u,v)=(-\frac 12, -1)$, does not escape to the corner at $(u,v)=(1,1)$ but remains bounded in the interior of the hidden domain $(-1,1)^2$.

\begin{figure}[ht]
    \begin{center}
        \includegraphics[width=10cm]{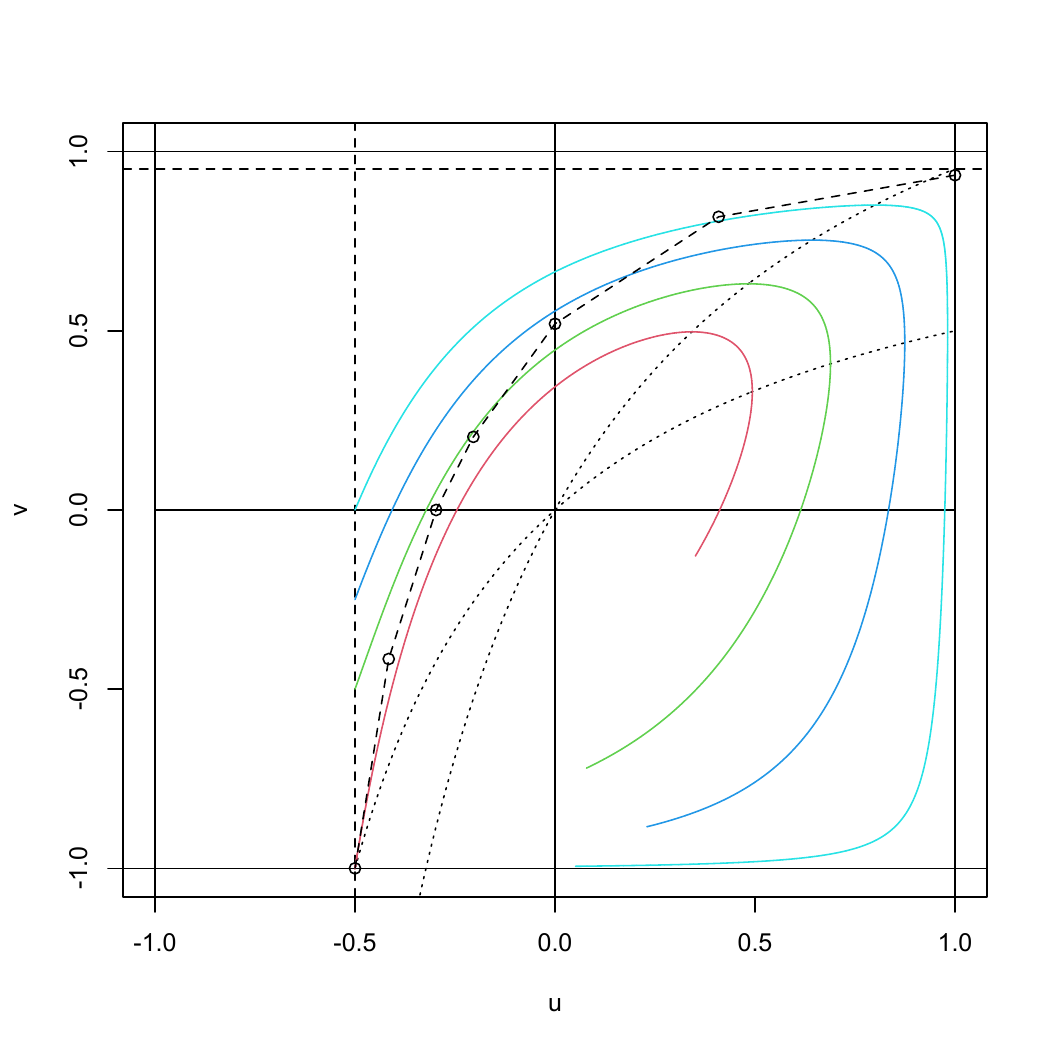}    
    \caption{Invariant region for System~\eqref{eq:uvexample1_hill} with $\kappa=1$ and $a_2=\frac{11}{10}$. Dashed lines: upper and left edges of the rectangle $R$, and the six line segments that form the upper boundary of the invariant region. Dotted lines: $u$- and $v$-nullclines. Coloured lines: trajectories.
    \label{fig:invariant}}
    \end{center}
\end{figure}

 We construct an invariant region inside $R$ ({\em i.e.}, below $v=\frac{2}{a_2+1}$); see Figure~\ref{fig:invariant}. Trajectories starting inside $R$ cannot leave $R$ through $u=1$ or $v=-1$, except potentially at the fixed point $(1,\frac{2}{a_2+1})$. Note that $u=1$ and $v=-1$ are themselves solution curves, $(1,-1)$ is a saddle point, so solutions from the interior of $R$ cannot reach it, and $(-\frac 12, -1)$ is a saddle point with $v=-1$ as the stable manifold, so solutions from the interior of $R$ cannot reach this point either. They also cannot leave $R$ through $u=-\frac 12$, since $\dot{u}>0$ there. This leaves the possibility of leaving $R$ through $v=\frac{2}{a_2+1}$. 

 However, we can show that the sequence of six straight line segments shown in Figure~\ref{fig:invariant} can only be crossed from left to right and thus the region of $R$ below this boundary is also invariant. The endpoints of these six line segments lie on the lines $v=-1$, $v=u$, $v=0$, $v=-u$, $u=0$, $v=2u$, and $u=1$, respectively, and are specified in Table~\ref{tab:invariant}. In each case, the slope of the line $v=m_i(u-u_i)+v_i$ is the direction of the vector field $\frac{dv}{du}$ at the lower (left) end of the segment. At $(u_0,v_0)=(-\frac 12, -1)$, one must use the slope of the unstable eigenvector of the Jacobian matrix at this point. The Jacobian matrix here is 
 \begin{eqnarray*}
    J(u_0,v_0)&=&\frac 14 \left[\begin{array}{cc}(1-u_0^2)(v_0-1) & (1-u_0^2)(u_0+1)\\ 0 & (-2v_0)(a_2u_0v_0-2u_0+v_0) \end{array}\right] \\
&=&\left[\begin{array}{cc}-\frac{3}{2} & -\frac{3}{8}\\ 0 & a_2\end{array}\right],
\end{eqnarray*} 
with unstable eigenvector
$$\xi=\left[\begin{array}{c} 3 \\ 12+8a_2 \end{array}\right].$$

 \begin{table}[h]
 \centering
 \begin{tabular}{|c|c|c|l|c|}
 \hline
 $i$ & $u_i$ & $v_i$ & on & $m_i$ \\
 \hline\hline
 0 & $-\frac 12$ & $-1$ & $v=-1$ & $\frac{12+8a_2}{3}$ \\ [3pt] \hline
 1 & $\frac{2-m_0}{2(m_0-1)}$ & $\frac{2-m_0}{2(m_0-1)}$ & $v=u$ & $\frac{a_2u_1-1}{u_1}$ \\ [3pt] \hline
 2 & $\frac{(m_1-1)u_1}{m_1}$ & 0 & $v=0$ & $\frac{2}{1-u_2^2}$ \\ [3pt] \hline
 3 & $\frac{m_2u_2}{m_2+1}$ & $-\frac{m_2u_2}{m_2+1}$ & $v=-u$ & $\frac{a_2u_3+3}{u_3+2}$ \\ [3pt] \hline
 4 & $0$ & $-u_3(m_3+1)$ & $u=0$ & $1-v_4^2$ \\ [3pt] \hline
 5 & $\frac{v_4}{1+v_4^2}$ & $\frac{2v_4}{1+v_4^2}$ & $v=2u$ & $\frac{2a_2u_5(1-2u_5)}{1-u_5^2}$ \\ [3pt] \hline
 6 & $1$ & $u_5(2-m_5)+m_5$ & $u=1$ & \\ [3pt] \hline
 \end{tabular}
 \caption{Line segments definining the upper (left) boundary of an invariant region in Figure~\ref{fig:invariant}.
 \label{tab:invariant}}
 \end{table}

For each line segment, $v=m_i(u-u_i) +v_i$, the vector field crosses from left to right if $\frac{dv}{du}<m_i$ on the segment (except at $(u_i,v_i)$, where $\frac{dv}{du}=m_i$ by construction). This inequality can be written 
$$\frac{dv}{du}=\frac{(1-v^2)(a_2uv-2u+v)}{(1-u^2)(uv-u+v)}<m_i\,,$$
or 
$$ (1-v^2)(a_2uv-2u+v)<m_i(1-u^2)(uv-u+v)\,,$$
since $1-u^2>0$ and $uv-u+v>0$ above the $u$-nullcline. Thus, the condition to check for each $i=0,\ldots ,5$ is 
\begin{eqnarray}
(1-[m_i(u-u_i)+v_i]^2)((a_2u+1)[m_i(u-u_i)+v_i]-2u)\hspace*{8mm} \nonumber\\
-m_i(1-u^2)((u+1)[m_i(u-u_i)+v_i]-u)<0
\label{eq:inv_condition}\end{eqnarray}
for $u_i<u<u_{i+1}$. This is a quartic polynomial in $u$, so it is possible to check exactly whether it holds for the specified $u$-interval. In fact, $u=u_i$ is a root of this quartic for each $i$ by construction, so in principle it can be reduced to $(u-u_i)$ times a cubic. Exact roots of this cubic can be found explicitly.

Note that although at each segment endpoint the slope of the next segment is exactly the slope of the vector field, a trajectory cannot go above the next line segment, as this would contradict the flow direction (the trajectory must be below any line segment of slope $>m_i$; then consider the limit as the slope of the segment approaches $m_i$).

For $a_2=\frac{11}{10}$, we confirmed that the condition~\eqref{eq:inv_condition} is satisfied for each of the six segments. Furthermore, $v_6<\frac{2}{a_2+1}$ ($v_6\approx 0.93432$ and $\frac{2}{a_2+1}=\frac{20}{21}\approx 0.95238$), so the boundary formed by the six line segments lies entirely in $R$, and thus forms the upper boundary of an invariant region. The trajectory from the entry point $(-\frac 12,-1)$ enters this region and cannot leave.
\section*{Acknowledgments}
RE acknowledges support from the Gran Sasso Science Institute for a visit during which much of this work was done. He also acknowledges support by Discovery Grant RGPIN-2017-04042 from the Natural Sciences and Engineering Research Council of Canada (NSERC).

NG acknowledges that his research was supported by funds from the Italian 
MUR (Ministero dell'Universit\`a e della Ricerca) within the PRIN 2017 Project 
``Discontinuous dynamical systems: theory, numerics and applications''. 
Moreover he acknowledges funding from Dipartimento di Eccellenza.

AA and NG are members of the INdAM-GNCS (Gruppo Nazionale di Calcolo Scientifico).


\end{document}